\documentclass[english]{article}
\usepackage[T1]{fontenc}
\usepackage[latin9]{inputenc}
\usepackage{amsthm}
\usepackage{amsmath}
\usepackage{amssymb}
\usepackage{graphicx}
\usepackage{esint}

\makeatletter
\numberwithin{equation}{section}
  \theoremstyle{definition}
  \newtheorem{defn}{\protect\definitionname}[section]
  \theoremstyle{plain}
  \newtheorem{thm}{\protect\theoremname}[section]
  \theoremstyle{remark}
  \newtheorem{rem}{\protect\remarkname}[section]
  \theoremstyle{plain}
  \newtheorem{prop}{\protect\propositionname}[section]
  \theoremstyle{plain}
  \newtheorem{lem}{\protect\lemmaname}[section]

\makeatother

\usepackage{babel}
  \providecommand{\definitionname}{Definition}
  \providecommand{\lemmaname}{Lemma}
  \providecommand{\propositionname}{Proposition}
  \providecommand{\remarkname}{Remark}
\providecommand{\theoremname}{Theorem}

\begin{document}

\title{The Uniqueness of Signature Problem in the Non-Markov Setting}

\author{H. Boedihardjo%
\thanks{The Oxford-Man Institute, University of Oxford, Eagle House, Walton
Well Road, Oxford OX2 6ED. Email: horatio.boedihardjo@oxford-man.ox.ac.uk %
} and X. Geng%
\thanks{Mathematical Institute, University of Oxford, Woodstock Road, Oxford
OX2 6GG and the Oxford-Man Institute, University of Oxford, Eagle
House, Walton Well Road, Oxford OX2 6ED. Email: xi.geng@maths.ox.ac.uk %
}}
\maketitle
\begin{abstract}
The goal of this paper is to simplify and strengthen the Le Jan-Qian
approximation scheme of studying the uniqueness of signature problem
to the non-Markov setting. We establish a general framework for a
class of multidimensional stochastic processes over $[0,1]$ under
which with probability one, the signature (the collection of iterated
path integrals in the sense of rough paths) is well-defined and determines
the sample paths of the process up to reparametrization. In particular,
by using the Malliavin calculus we show that our method applies to
a class of Gaussian processes including fractional Brownian motion
with Hurst parameter $H>1/4$, the Ornstein-Uhlenbeck process and
the Brownian bridge.
\end{abstract}

\section{Introduction}

The set of continuous paths forms a semigroup with involution, with
the group operation and involution given by the concatenation and
reversal of paths. In as early as 1954, K.T. Chen \cite{Chen54} observed
that the map sending a bounded variation path $x:\left[0,1\right]\rightarrow\mathbb{R}^{d}$
to the formal series 
\[
1+\int_{0}^{1}dx_{s}^{i}X_{i}+\int_{0}^{1}\int_{0}^{s_{2}}dx_{s_{1}}^{i}dx_{s_{2}}^{j}X_{i}X_{j}+\ldots,
\]
where $X_{i}$ ($i=1,\cdots,d$) are indeterminates and $x^{i}$
denote the $i$-th coordinate component of $x$, is a homomorphism
from the semigroup of paths to the algebra of non-commutative formal
power series. Unfortunately, this map is not injective. The homomorphism
property of the map implies that any path concatenated with its reversal
will be mapped to the trivial formal series. It seems however that
the map is essentially injective if we restrict our attention to paths
that ``do not track back along itself''. Indeed, Chen himself \cite{Chen uniqueness}
proved that the map is injective on the space of regular, irreducible
paths. In \cite{tree like}, B. Hambly and T. Lyons extended Chen's
result to the space of paths with bounded variation and introduced
the notion of \textit{tree-like paths} to describe paths that track
back along itself. In particular, they proved that the formal series
corresponding to a path, which they called the signature of a path,
is trivial if and only if the path is tree-like. 

Aside from its interesting algebraic properties, the map also gains
attention through the fundamental role it plays in rough path theory.
In \cite{Young}, L.C. Young defined the Stieltjes type integral $\int_{0}^{1}y_{t}dx_{t}$
in terms of a Riemann sum when $x_{\cdot}$ and $y_{\cdot}$ have
finite $p$ and $q$-variation respectively, where $\frac{1}{p}+\frac{1}{q}>1$.
In particular, it allows us to define, for a Lipschitz one form $\phi$,
the integral $\int_{0}^{1}\phi\left(x_{t}\right)dx_{t}$ when $x$
is a multidimensional path with finite $p$-variation for $p<2$.
In the same paper, Young gave an example where the integral $\int_{0}^{1}\phi\left(x_{t}\right)dx_{t}$
defined using Riemann sum would diverge when $x$ has only finite
$2$-variation. In other words, the Stieltjes integration map $x\rightarrow\int_{0}^{1}\phi\left(x_{t}\right)dx_{t}$
does not have a closable graph with respect to $p$-variation if $p\geqslant2$.
The seemingly insurmountable $p=2$ barrier, at least in the deterministic
setting, is to remain for another sixty years. In \cite{MR1654527},
T. Lyons showed that the Stieltjes integration map would have a closable
graph with respect to the $p$-variation metric if the path $x$ takes
value in a step-$\left\lfloor p\right\rfloor $ nilpotent Lie group.
He called these paths \textit{weakly geometric $p$-rough paths}.
The first step in the construction of such integral is to define the
signature for weakly geometric $p$-rough paths. The integration of
one forms against such paths is then defined via polynomial approximations.
There has been excellent progress in extending the rough path theory
to even more general paths in, for example, the work of M. Gubinelli
\cite{Gub10}, M. Hairer and D. Kelly \cite{HK13}, etc.

From a theoretical standpoint, once Lyons defined the signature for
weakly geometric rough paths, it is natural to ask if the signature
of a weakly geometric rough path determines the path uniquely up to
tree-like equivalence as in the bounded variation case. From a practical
point of view, there has also been works done in, for example, D.
Levin, T. Lyons and H. Ni \cite{LLN13} on analyzing time series data
using the signature map. The justification of their method implicitly
used the fact that the map from a path to its signature is injective
in some sense. The solution of this long standing open problem in
rough path theory is contained in the very recent work by H. Boedihardjo,
X. Geng, T. Lyons and D. Yang \cite{BGLY14}. 

There has also been exciting progress of the problem in the probabilistic
setting. In \cite{LQ12}, Y. Le Jan and Z. Qian proved that with probability
one, the Stratonovich signatures of Brownian motion determine the
Brownian sample paths. Their strategy, in particular the approximation
scheme constructed in the proof, was originated from the study of
cyclic cohomology in algebraic topology. The proof relies heavily
on the strong Markov property and the potential theory for the Laplace
operator. This result was then extended to hypoelliptic diffusions
by X. Geng and Z. Qian \cite{GQ13}. Similar results were also established
for Chordal SLE$_{\kappa}$ curves with $\kappa\leqslant4$ by H.
Boedihardjo, H. Ni and Z. Qian \cite{BNQ12}. 

It should be pointed out that in the probabilistic setting, the result
of Le Jan and Qian is stronger than the general deterministic result
in \cite{BGLY14}, as it not only gives the injectivity but also gives
an explicit way of how the sample path can be reconstructed from its signature
outside a null set in the path space. In the deterministic setting,
such reconstruction was studied by T. Lyons and W. Xu \cite{LX14}
for $C^{1}$-paths via symmetrization, and by H. Boedihardjo and X.
Geng \cite{BG13} for planar Jordan curves with finite $p$-variation
for $1\leqslant p<2$ via Fourier transform. A general inversion scheme
for the signature of a weakly geometric rough path remains a significant
open problem in rough path theory.

The main purpose of this paper is to simplify and strengthen the method
of Le Jan and Qian to include a class of non-Markov processes. In
particular, we shall establish the almost-sure uniqueness of signature
(up to reparametrization) for a class of Gaussian processes including
fractional Brownian motion with Hurst parameter $H>1/4$, the Ornstein-Uhlenbeck
process and the Brownian bridge. More importantly, our technique also
yields an explicit inversion scheme for the signature of sample paths.The
fundamental difficulty in exploiting the idea of Le Jan and Qian lies
in the unavailability of those probabilistic and analytic tools arising
from the strong Markov property and the potential theory which were
used in their proof. The key of getting around this difficulty is to understand 
the pathwise nature of the problem and to find methods to analyze pathwisely based
on techniques from rough path theory. In the fundamental example of Gaussian processes, 
the key idea is to make use of the structure of the Cameron-Martin space and to
apply local regularity results for Gaussian functionals from the Malliavin calculus.

The well-definedness of signature when the sample paths of the process
have finite $p$-variation for $p\geqslant1$ are well studied in
probability literatures. For instance, it was shown by L. Coutin and
Z. Qian \cite{CQ00} that with probability one, the sample paths of
fractional Brownian motion with Hurst parameter $H>\frac{1}{4}$ can
be lifted canonically as geometric rough paths. Moreover, it is believed
that no such canonical lift exists for $H\leqslant\frac{1}{4}$. There
are similar results for lots of interesting stochastic processes,
such as martingales, Markov processes, Gaussian processes, solutions
to Gaussian rough differential equations, SLE$_{\kappa}$ curves with
$\kappa\leqslant4$ etc., under certain regularity conditions. See for
example \cite{FV10}.

In establishing our main result, we shall state explicitly under what
conditions on the law of the process would the almost-sure uniqueness
of signature hold. We hope that this provides a general framework
for solving the almost-sure uniqueness of signature problem for other
interesting processes. Note that our result is \textit{not }a direct
corollary of the result in \cite{BGLY14}, since it is highly nontrivial to prove
the existence of a null set outside which no two paths can be tree-like deformation of each other.

\section{Preliminaries on Rough Path Theory}

We first recall some basic notions from rough path theory, which we
will use throughout the rest of this paper. 

Let $T\left(\left(\mathbb{R}^{d}\right)\right)$ denote the infinite
dimensional tensor algebra over $\mathbb{R}^{d}$. Let $\pi_{k}$
denote the projection map from $T\left(\left(\mathbb{R}^{d}\right)\right)$
to $\left(\mathbb{R}^{d}\right)^{\otimes k}$ and $\pi^{\left(k\right)}$
denote the projection map from $T\left(\left(\mathbb{R}^{d}\right)\right)$
to the truncated $k$-th tensor algebra 
\[
T^{k}\left(\mathbb{R}^{d}\right):=\oplus_{j=0}^{k}\left(\mathbb{R}^{d}\right)^{\otimes j}.
\]
Here we shall equip $\left(\mathbb{R}^{d}\right)^{\otimes k}$ with
the Euclidean norm by identifying it with $\mathbb{R}^{d^{k}}$. Let
$\triangle:=\left\{ \left(s,t\right):0\leqslant s\leqslant t\leqslant1\right\} $
be the standard $2$-simplex.
\begin{defn}
A \textit{multiplicative functional of degree} $N\in\mathbb{N}$ is
a map $\mathbf{X}:\triangle\rightarrow T^{N}\left(\mathbb{R}^{d}\right)$
satisfying the following so-called Chen's identity: 
\[
\mathbf{X}_{s,u}\otimes\mathbf{X}_{u,t}=\mathbf{X}_{s,t},\ \forall0\leqslant s\leqslant u\leqslant t\leqslant1.
\]
Let $\mathbb{\mathbf{X}},\mathbf{Y}$ be two multiplicative functionals
of degree $N.$ For $p\geqslant1$, define 
\[
d_{p}\left(\mathbf{X},\mathbf{Y}\right)=\max_{1\leqslant i\leqslant N}\sup_{\mathcal{P}}\left(\sum_{l}\left|\pi_{i}\left(\mathbf{X}_{t_{l-1},t_{l}}-\mathbf{Y}_{t_{l-1},t_{l}}\right)\right|^{\frac{p}{i}}\right)^{\frac{i}{p}},
\]
where the supremum is taken over all possible finite partitions of
$[0,1]$. $d_{p}$ is called the \textit{p-variation metric}. If $d_{p}\left(\mathbf{X},\mathbf{1}\right)<\infty$
where $\mathbf{1}=(1,0,\cdots,0)$, we then say that $\mathbf{X}$
has \textit{finite p-variati}on. A multiplicative functional of degree
$\lfloor p\rfloor$ with finite $p$-variation is called a \textit{p-rough
path}.
\end{defn}

The following so-called Lyons' extension theorem (see \cite{MR1654527})
says that the signature of a $p$-rough path is well defined. 
\begin{thm}
\label{Lyons' Extension Theorem}For $p\geqslant1$, let $\mathbf{X}$
be a $p$-rough path. Then there exists a unique multiplicative functional
$S\left(\mathbf{X}\right):\triangle\rightarrow T\left(\left(\mathbb{R}^{d}\right)\right)$
such that $\pi^{\left(N\right)}\left(S\left(\mathbf{X}\right)\right)$
has finite $p$-variation for each $N\in\mathbb{N}$ and 
\[
\pi^{\left(\left\lfloor p\right\rfloor \right)}\left(S\left(\mathbf{X}\right)\right)=\mathbf{X}.
\]

\end{thm}

\begin{defn}
$S\left(\mathbf{X}\right)_{0,1}\in T\left(\left(\mathbb{R}^{d}\right)\right)$
is called the \textit{signature} of the $p$-rough path $\mathbf{X}$. 
\end{defn}

If $x:\left[0,1\right]\rightarrow\mathbb{R}^{d}$ is a path with finite
$p$-variation for some $1\leqslant p<2$, then as a $p$-rough path
no higher levels of $x$ are needed and we can express the signature
of $x$ explicitly as 
\[
S\left(x\right)_{0,1}=\left(1,\int_{0<s_{1}<1}dx_{s_{1}},\ldots,\int_{0<s_{1}<\cdot<s_{n}<1}dx_{s_{1}}\otimes\ldots\otimes dx_{s_{n}},\ldots\right),
\]
where the iterated integrals are defined in the sense of Young.

A fundamental result in rough path theory, proved by Lyons \cite{MR1654527},
is the continuity of rough path integrals and the solution map for
rough differential equations with respect to the driving path under
the $p$-variation metric.

There is a special class of rough paths called geometric rough paths.
They are the simplest and very natural examples of rough paths which
we can define path integrals against one forms.
\begin{defn}
Given $p\geqslant1$. Let $G\Omega_{p}\left(\mathbb{R}^{d}\right)$
denote the completion of the set 
\[
\left\{ S_{\lfloor p\rfloor}(x):=\pi^{\left(\left\lfloor p\right\rfloor \right)}\left(S(x)\right):\; x\mbox{ has bounded total variation}\right\} 
\]
with respect to the $p$-variation metric $d_{p}$. $G\Omega_{p}(\mathbb{R}^{d})$
is called the space of \textit{geometric p-rough paths.}
\end{defn}

In \cite{CQ00}, Coutin and Qian showed that under certain conditions
on the decorrelation of the increment of a Gaussian process, with
probability one the lifting of the dyadic piecewise linear interpolation
of the Gaussian sample paths in $G\Omega_{p}\left(\mathbb{R}^{d}\right)$
is a Cauchy sequence under the $p$-variation metric. In \cite{FV10},
P. Friz and N. Victoir extended this result to a larger class of Gaussian
processes under certain regularity condition on the covariance function.
Moreover, they showed that the lifting of any sequence of piecewise
linear interpolation of the Gaussian sample paths in $G\Omega_{p}$
converges to the same limit. From here onwards, this limit will be
known as the \textit{canonical lifting} of the Gaussian process in
$G\Omega_{p}\left(\mathbb{R}^{d}\right)$. A fundamental example of
these results is fractional Brownian motion with Hurst parameter $H>1/4$.
It follows from Theorem \ref{Lyons' Extension Theorem} that the signature
$S(x)_{0,1}\in T(\mathbb{R}^{d})$ of fractional Brownian motion with
$H>1/4$ is well-defined for almost surely through the canonical lifting.

A detailed study on the geometric rough path nature of many interesting
and important stochastic processes can be found in \cite{FV10}.

\section{Main Results}

In this section we are going to state main results of the paper and
illustrate the idea of proofs.

Let $X=\{X_{t}:\ t\in[0,1]\}$ be a $d$-dimensional continuous stochastic
process starting at the origin, where $d\geqslant2$. We will always
assume that $X$ is realized on the path space $(W,\mathcal{B}(W),\mathbb{P}),$where
$W$ is the space of $\mathbb{R}^{d}$-valued continuous paths over
$[0,1]$ starting at the origin, $\mathcal{B}(W)$ is the Borel $\sigma$-algebra
over $W,$ and $\mathbb{P}$ is the law of $X.$

In the rest of this paper, we will make the following assumptions
on the law $\mathbb{P}.$ 

\textbf{Assumption (A)}: There exists a $\mathbb{P}$-null set $\mathcal{N}_{0}$
and a map $S:W\backslash\mathcal{N}_{0}\rightarrow C\left(\triangle;T\left(\left(\mathbb{R}^{d}\right)\right)\right)$,
such that for each $x\in W\backslash\mathcal{N}_{0}$ and $(s,t)\in\Delta$,
$\pi_{1}\left(S\left(x\right)_{s,t}\right)=x_{t}-x_{s}$ and $S\left(x\right)$
is the multiplicative extension of some geometric rough path $\mathbf{X}$
(see Theorem \ref{Lyons' Extension Theorem}). We will call such a
map $S$ a \textit{$\mathbb{P}$-almost sure lifting}. The integrals
with respect to $x$ will then be defined as integrating against the
geometric rough path $\mathbf{X}$. 

\textbf{Assumption (B)}: For any $0<t<1,$ the law of $x_{t}$ is
absolutely continuous with respect to the Lebesgue measure.

\textbf{Assumption (C)}: For any open cube $H\subset\mathbb{R}^{d}$,
there exists a differential one form $\phi=\sum_{i=1}^{d}\phi_{i}dx^{i}$
supported on the closure of $H$, such that for any $0\leqslant s<t\leqslant1,$
if we let 
\begin{equation}
A_{s,t}^{H}=\{x\in W:\ \mbox{there exists some \ensuremath{u\in(s,t)}}\ \mbox{such that \ensuremath{x_{u}\in H}}\},\label{travel through the interior}
\end{equation}
then 
\[
\mathbb{P}\left(\left\{ x\in W:\ \int_{s}^{t}\phi(dx_{u})=0\right\} \cap A_{s,t}^{H}\right)=0.
\]
Here $\int_{s}^{t}\phi(dx_{u})=\sum_{i=1}^{d}\int_{s}^{t}\phi_{i}(x_{u})dx_{u}^{i}$
is defined in the sense of rough paths according to Assumption (A).
\begin{rem}
As we've mentioned before, Assumption (A) is quite natural for a large
class of stochastic processes. Assumption (B) is also verified for
most of these processes, e.g., hypoelliptic diffusions, Gaussian processes,
solutions to hypoelliptic rough differential equations driven by Gaussian
processes, etc. These examples are well studied in \cite{FV10}. Assumption
(C) suggests certain kind of nondegeneracy for sample paths of the
process, which is essential for the recovery of a path from its signature
in our setting. By a closer look at Assumption (C), it actually excludes
the possibility of the sample paths being tree-like. Therefore, with
probability one the sample paths are already ``reduced'' paths in
the tree-like equivalence classes and it is natural to expect an inversion
scheme for the signature in our setting (see \cite{BGLY14}, \cite{tree like}
for the notion of tree-like paths). This is the main goal of the present
paper.

In the last section, as a fundamental example we will show that these
assumptions are all verified for a class of Gaussian processes including
fractional Brownian motion with Hurst parameter $H>1/4$, the Ornstein-Uhlenbeck
process and the Brownian bridge.
\end{rem}

Since we aim at recovering a path up to reparametrization from its
signature, we first give the definition of reparametrization.
\begin{defn}
\label{reparametrization}A \textit{reparametrization} is a continuous,
strictly increasing map $\sigma:\ [0,1]\rightarrow[0,1]$ with $\sigma(0)=0$
and $\sigma(1)=1.$ The group of reparametrizations is denoted by
$\mathcal{R}.$
\end{defn}

Now we are in a position to state our main results.
\begin{thm}
\label{main thm general framework}Assume that the law $\mathbb{P}$
of the stochastic process satisfies Assumption (A), (B) and (C). Let
$S$ be the $\mathbb{P}$-almost sure lifting as in Assumption (A).
Then there exists a $\mathbb{P}$-null set $\mathcal{N},$ such that
for any $x,x'\in\mathcal{N}^{c}$, if $S(x)_{0,1}=S(x')_{0,1}$, then
there exists some $\sigma\in\mathcal{R},$ such that 
\[
x{}_{t}=x'_{\sigma(t)},\ \forall t\in[0,1].
\]

\end{thm}

As a fundamental example, we will prove the following result for a
class of Gaussian processes satisfying conditions to be specified
later on in the final section. 
\begin{thm}
\label{main thm fBM}Let $\mathbb{P}$ be the law of a Gaussian process
satisfying conditions specified in Section 6. Then $\mathbb{P}$ satisfies
Assumption (A), (B), (C). In particular, the result holds for fractional
Brownian motion with Hurst parameter $H>1/4$, the Ornstein-Uhlenbeck
process and the Brownian bridge.
\end{thm}

Before going into the mathematical proofs, we first describe the strategy
informally. The approximation scheme we are going to use is an adaptation
from the work of Le Jan and Qian \cite{LQ12}. However, the main difficulties
are in the development of each step in the non-Markov setting, which
will be clear in the detailed proofs.

\textit{Step One.} Prove that if two paths have the same signature,
then the iterated integrals of the paths along any finite sequence
of smooth one forms are the same. Following \cite{LQ12}, these iterated
integrals along one forms will be called \textit{extended signatures.}

\textit{Step Two.} Decompose the Euclidean space $\mathbb{R}^{d}$
into disjoint identical open cubes with small tunnels between them.
For each such cube, we define a differential one form supported on
the closure of the cube according to Assumption (C).

\textit{Step Three.} Show that, for each path $x$ outside a $\mathbb{P}$-null
set, the ordered sequence of cubes visited by $x$ corresponds to
the unique maximal sequence of differential one forms along which
the extended signature of $x$ is nonzero. This together with step
one allows us to recover the ordered sequence of cubes visited by
$x$ from its signature. 

\textit{Step Four.} Construct a polygonal approximation of $x$ by
joining the centers of cubes visited by $x$ in order. This polygonal
path will be parametrized so that it is at the center of the cube
at the time when the cube is first visited by $x$. Show that with
probability one, as the size of cubes tends to zero, the polygonal
path converges to the original path $x$ under the uniform topology.

\textit{Step Five.} Since the signature is invariant under the reparametrization
of the path, it is not possible to recover the exact visit times of
the cubes. If two paths have the same signature, then the corresponding
polygonal paths constructed in (3) are only equal up to a reparametrization.
Therefore, we need to introduce a variant of the Fréchet distance
on $W$ measuring the distance of two paths modulo parametrization.
We should also prove that outside a $\mathbb{P}$-null set this is
indeed a metric. It will then follow from step four that if two paths
$x$ and $x^{\prime}$ have the same signature, their corresponding
approximation paths converge to the same limit under this metric,
which will imply that $x$ and $x'$ are equal up to a reparametrization.

For the Gaussian case, Assumption (A) is verified from \cite{FV10}
and Assumption (B) is trivial by definition. By using the Malliavin calculus, 
for each open cube $H$ we will explicitly construct a differential
one form $\phi$ supported on $\overline{H}$ such that the functional
$x\rightarrow\int_{s}^{t}\phi(dx_{u})$ has a density conditioned
on the set $A_{s,t}^{H}$. This certainly verifies Assumption (C).

\section{Signature Determines Extended Signatures}

Starting from this section, we are going to develop the detailed proofs
of our main results.

As the first step, here we will prove that if two sample paths as
geometric rough paths have the same signatures, then they have the
same extended signatures. Note that the signatures and extended signatures
are well-defined for $\mathbb{P}$-almost surely according to Assumption
(A). For the general theory of integration along one forms against
rough paths, see for example \cite{FV10}, \cite{LyonsQ}. 

From now on, for a geometric rough path $\mathbf{X}$ and a finite
sequence $(\phi^{1},\cdots,\phi^{n})$ of differential one forms $\phi^{1},\ldots,\phi^{n}$,
we will use $\left[\phi^{1},\cdots,\phi^{n}\right]_{0,1}\left(x\right)$
to denote the iterated path integral $\int_{0}^{1}\cdots\int_{0}^{s_{2}}\phi^{1}\left(d\mathbf{X}_{s_{1}}\right)\cdots\phi^{n}\left(d\mathbf{X}_{s_{n}}\right)$,
where $x_{\cdot}:=\pi_{1}(S(\mathbf{X})_{0,\cdot})$ is the first
level path of $\mathbf{X}.$ A simple way of understanding this integral
is via 
\[
\int_{0}^{1}\ldots\int_{0}^{s_{2}}\phi^{1}\left(d\mathbf{X}_{s_{1}}\right)\ldots\phi^{n}\left(d\mathbf{X}_{s_{n}}\right)=\lim_{k\rightarrow\infty}\int_{0}^{1}\cdots\int_{0}^{s_{2}}\phi^{1}(dx_{s_{1}}^{(k)})\cdots\phi^{n}(dx_{s_{n}}^{(k)}),
\]
where by the definition of geometric rough paths $x^{(k)}$ is a sequence
of paths with bounded total variation whose lifting converges to $\mathbf{X}$
under the $p$-variation metric. Sometimes we will also use the notation
$\int_{0}^{1}\ldots\int_{0}^{s_{2}}\phi^{1}\left(dx_{s_{1}}\right)\ldots\phi^{n}\left(dx_{s_{n}}\right)$
to denote the path integral. Note that the ordering of $(\phi^{1},\cdots,\phi^{n})$
is noncommutative in this notation. 

Now we have the following result.
\begin{prop}
\label{signature determines extended signatures} Given $p\geqslant1,$
let $\mathbf{X},\mathbf{X}^{\prime}\in G\Omega_{p}\left(\mathbb{R}^{d}\right)$
be two geometric $p$-rough paths. Suppose that $\phi^{1},\ldots,\phi^{n}$
are $n$ compactly supported $C^{\infty}$-one forms. If $S\left(\mathbf{X}\right)_{0,1}=S\left(\mathbf{X}^{\prime}\right)_{0,1}$,
then 
\[
\left[\phi^{1},\ldots,\phi^{n}\right]_{0,1}\left(x\right)=\left[\phi^{1},\ldots,\phi^{n}\right]_{0,1}\left(x^{\prime}\right),
\]
where $x$ and $x'$ are the first level paths of $\mathbf{X}$ and
$\mathbf{X}'$ respectively.
\end{prop}
To prove Proposition \ref{signature determines extended signatures},
first notice that the case of polynomial one forms follows immediately
from integration by parts and the shuffle product formula for the
signature (see \cite{LQ12}, p. 4 and \cite{LCL07}, Theorem 2.15). 
\begin{lem}
\label{signature determines extended signatures as lemma}Let $\phi_{1},\ldots,\phi_{n}$
be $n$ polynomial one forms. Let $x$ be a continuous path with bounded
total variation with $x_{0}=0$. Then there exists a linear functional
$f$ on $T\left(\mathbb{R}^{d}\right)$ such that 
\begin{equation}
\left[\phi^{1},\ldots,\phi^{n}\right]_{0,1}\left(x\right)=f\left(S\left(x\right)_{0,1}\right).\label{eq:polynomial functional}
\end{equation}

\end{lem}

Proposition \ref{signature determines extended signatures} then follows
from polynomial approximations.

\begin{proof}[Proof of Proposition \ref{signature determines extended signatures}]
We write $\phi^{i}$ as $\phi^{i}=\sum_{j=1}^{d}\phi_{j}^{i}\left(x\right)dx^{j}$.
Let $K$ be a compact neighborhood of $x([0,1])\cup x'([0,1])$. According
to \cite{BBL02}, Theorem 1, for each $\alpha>0$ and each $j$, there
exists a polynomial sequence $\phi_{j}^{i\left(m\right)}$ such that
\[
\sup_{K}\left|D^{\alpha}\left(\phi_{j}^{i}-\phi_{j}^{i(m)}\right)\right|\rightarrow0
\]
as $m\rightarrow\infty$. Let $\phi^{i\left(m\right)}\left(x\right)=\sum_{j=1}^{d}\phi_{j}^{i\left(m\right)}\left(x\right)dx^{j}$.
As $\mathbf{X}\in G\Omega_{p}\left(\mathbb{R}^{d}\right)$, by definition
there exists a sequence $x^{(k)}$ of paths with bounded total variation,
such that $d_{p}\left(S_{\left\lfloor p\right\rfloor }\left(x^{\left(k\right)}\right),\mathbf{X}\right)\rightarrow0$
as $k\rightarrow\infty.$ Since the integration map $\left(\phi,\mathbf{X}\right)\rightarrow\int_{0}^{1}\phi(d\mathbf{X}_{t})$
is jointly continuous under the Lip$\left(\alpha\right)$ and $p$-variation
norms whenever $\alpha>p+1$ (see \cite{FV10}, Theorem 10. 47), we
have
\[
\left[\phi^{1\left(m\right)},\ldots,\phi^{n\left(m\right)}\right]_{0,1}\left(x^{\left(k\right)}\right)\rightarrow\left[\phi^{1},\ldots,\phi^{n}\right]_{0,1}\left(x\right),
\]
as $m,k\rightarrow\infty$. Now the result follows from Lemma \ref{signature determines extended signatures as lemma}.\end{proof}

\section{The Strengthened Le Jan-Qian Approximation Scheme and the Uniqueness
of Signature}

Now fix $\varepsilon,\delta>0$ with $\delta<<\varepsilon$. 

For any integer point $z=(z^{1},z^{2},\ldots,z^{d})\in\mathbb{Z}^{d},$
let $H_{z}^{\varepsilon,\delta}$ be the open cube centered at $\varepsilon z$
with edges of length $\varepsilon-\delta.$ In other words,
\[
H_{z}^{\varepsilon,\delta}=\left\{ x\in\mathbb{R}^{d}:\ \left|x^{i}-\varepsilon z^{i}\right|<\frac{\varepsilon-\delta}{2},\ \forall i=1,\cdots,d\right\} .
\]
Geometrically, the space $\mathbb{R}^{d}$ is divided into disjoint
identical open cubes and small closed tunnels.

For any $x\in W$ and $k\geqslant1,$ define recursively 
\[
\tau_{k}^{\varepsilon,\delta}=\inf\left\{ t\in\left[\tau_{k-1}^{\varepsilon,\delta},1\right]:\ x_{t}\in\bigcup_{z\neq\boldsymbol{m}_{k-1}^{\varepsilon,\delta}}H_{z}^{\varepsilon,\delta}\right\} ,
\]
and $\boldsymbol{m}_{k}^{\varepsilon,\delta}$ to be the integer point
$z\in\mathbb{Z}^{d}$ such that 
\[
x_{\tau_{k}^{\varepsilon,\delta}}\in H_{z}^{\varepsilon,\delta},
\]
where $\tau_{0}^{\varepsilon,\delta}=0,\ \boldsymbol{m}_{0}^{\varepsilon,\delta}=0\in\mathbb{Z}^{d}$.
Let 
\[
N^{\varepsilon,\delta}=\sup\left\{ k\geqslant1:\ \tau_{k}^{\varepsilon,\delta}<1\right\} ,
\]
where $\sup\emptyset:=0.$ The sequence $\left\{ \tau_{k}^{\varepsilon,\delta}\right\} $
records the successive visit times of the open cubes by the path,
the sequence $\left\{ \boldsymbol{m}_{k}^{\varepsilon,\delta}\right\} $
records the cubes visited in order, and $N^{\varepsilon,\delta}$
records the total number of cubes visited. Note that revisit of the
same cube after visiting some other cubes counts, but revisit before
visiting any other cube does not count. By continuity and compactness,
it is easy to see that for any $x\in W,$ $0\leqslant N^{\varepsilon,\delta}<\infty.$

Here and thereafter, for notation simplicity we drop the dependence
on $x$ for these random variables on $W.$
\begin{rem}
It is important to use the open cubes instead of the closed ones,
as we are only interested in the case when a path $x$ travels through
the interior of a cube. Hence these $\tau_{k}^{\varepsilon,\delta}$
are not stopping times with respect to the natural filtration.
\end{rem}

For each cube $H_{z}^{\varepsilon,\delta}$, let $\phi_{z}^{\varepsilon,\delta}$
be the differential one form given in Assumption (C). In particular,
$\phi_{z}^{\varepsilon,\delta}$ is supported on the closure of $H_{z}^{\varepsilon,\delta}$,
and $\phi_{z}^{\varepsilon,\delta}=0$ on $\partial H.$

\subsection{Recovery of Cubes Visited in Order by Using the Extended Signature}

Let $\mathcal{W}_{m}$ ($m\geqslant0$) be the set of words $(z_{0}=0,z_{1},\cdots,z_{m})$
with $z_{i}\neq z_{i+1},$ $z_{i}\in\mathbb{Z}^{d}$, and let $\mathcal{W}=\bigcup_{m\geqslant0}\mathcal{W}_{m}$.
Elements of $\mathcal{W}$ are called \textit{admissible words}. 

For $w=(z_{0},z_{1},\cdots,z_{m})\in\mathcal{W},$ define 
\[
E_{w}^{\varepsilon,\delta}=\left\{ x\in W:\ N^{\varepsilon,\delta}=m,\ \boldsymbol{m}_{k}^{\varepsilon,\delta}=z_{k},\ k=0,\cdots,m\right\} .
\]
It follows that $W$ can be written as the disjoint union $W=\bigcup_{w\in\mathcal{W}}E_{w}^{\varepsilon,\delta}$. 

Now we have the following result.
\begin{lem}
\label{recovering the ordered sequence of boxes}For any $m\geqslant0,$
if $w=(z_{0}=0,\cdots,z_{m})\in\mathcal{W}_{m}$ and $x\in E_{w}^{\varepsilon,\delta},$
then 

(1) 
\begin{equation}
\left[\phi_{z_{0}}^{\varepsilon,\delta},\cdots,\phi_{z_{m}}^{\varepsilon,\delta}\right]_{0,1}(x)=\prod_{i=1}^{m+1}\int_{\tau_{i-1}^{\varepsilon,\delta}}^{\tau_{i}^{\varepsilon,\delta}}\phi_{z_{i-1}}^{\varepsilon,\delta}(dx_{t}),\label{breaking into product}
\end{equation}
where $\tau_{m+1}^{\varepsilon,\delta}=1$ by definition since $x\in E_{w}^{\varepsilon,\delta}.$ 

(2) For any $w'=(z_{0},z'_{1},\cdots,z'_{n})\in\mathcal{W}_{n}$ with
$n>m,$ 
\[
\left[\phi_{z_{0}}^{\varepsilon,\delta},\cdots,\phi_{z'_{n}}^{\varepsilon,\delta}\right]_{0,1}(x)=0.
\]

(3) For any $w'=(z_{0},z'_{1},\cdots,z'_{m})$ with $w'\neq w,$ 
\[
\left[\phi_{z_{0}}^{\varepsilon,\delta},\cdots,\phi_{z'_{m}}^{\varepsilon,\delta}\right]_{0,1}(x)=0.
\]
\end{lem}
\begin{proof}
We prove this result by induction on $m$. 

If $m=0,$ assume that $x\in E_{(z_{0})}^{\varepsilon,\delta}.$ Then
(1) and (3) are trivial. To see (2), let $w'=(z_{0},z'_{1},\cdots,z'_{n})\in\mathcal{W}_{n}$
with $n>0.$ Since $w'$ is an admissible word, there is some $0<k\leqslant n$
such that $x$ does not visit the open cube $H_{z'_{k}}^{\varepsilon,\delta}$
and the corresponding extended signature is zero by definition (here
we have implicitly used the definition of extended signatures of geometric
rough paths and the joint continuity of the integration map with respect
to the one forms and the driving path). If $m=1,$ assume that $w=(z_{0},z_{1})\in\mathcal{W}_{1}$
and $x\in E_{w}^{\varepsilon,\delta}.$ Then (3) follows by the same
argument as before. To see (1), first we have 
\begin{eqnarray*}
\left[\phi_{z_{0}}^{\varepsilon,\delta},\phi_{z_{1}}^{\varepsilon,\delta}\right]_{0,1}(x) & = & \int_{0}^{1}\left[\phi_{z_{0}}^{\varepsilon,\delta}\right]_{0,t}(x)\phi_{z_{1}}^{\varepsilon,\delta}(dx_{t})\\
 & = & \int_{\tau_{1}^{\varepsilon,\delta}}^{1}\left[\phi_{z_{0}}^{\varepsilon,\delta}\right]_{0,t}(x)\phi_{z_{1}}^{\varepsilon,\delta}(dx_{t}),
\end{eqnarray*}
since $\phi_{z_{1}}^{\varepsilon,\delta}$ is supported in $H_{z_{1}}^{\varepsilon,\delta}.$
Moreover, if $\tau_{1}^{\varepsilon,\delta}\leqslant t\leqslant1,$
then 
\[
\left[\phi_{z_{0}}^{\varepsilon,\delta}\right]_{0,t}(x)=\left[\phi_{z_{0}}^{\varepsilon,\delta}\right]_{0,\tau_{1}^{\varepsilon,\delta}}(x),
\]
since $\phi_{z_{0}}^{\varepsilon,\delta}$ is supported in $H_{z_{0}}^{\varepsilon,\delta}.$
Therefore,
\[
\left[\phi_{z_{0}}^{\varepsilon,\delta},\phi_{z_{1}}^{\varepsilon,\delta}\right]_{0,1}(x)=\left(\int_{0}^{\tau_{1}^{\varepsilon,\delta}}\phi_{z_{0}}^{\varepsilon,\delta}(dx_{t})\right)\left(\int_{\tau_{1}^{\varepsilon,\delta}}^{1}\phi_{z_{1}}^{\varepsilon,\delta}(dx_{t})\right)
\]
and (1) follows. If $w'=(z_{0},z'_{1},\cdots,z'_{n})\in\mathcal{W}_{n}$
with $n>1,$ there are two case. The first case is that there is some
$0<k\leqslant n$ such that $z'_{k}$ is different from $z_{0}$ and
$z_{1}.$ In this case (2) follows by the same argument as before.
The second case is 
\[
w'=(z_{0},z_{1},z_{0},z_{1},\cdots,z'_{n}),
\]
where $n>1$ and $z'_{n}$ is either $z_{0}$ or $z_{1}$. If $z'_{n}=z_{0}$,
then 
\[
\left[\phi_{z_{0}}^{\varepsilon,\delta},\cdots,\phi_{z'_{n}}^{\varepsilon,\delta}\right]_{0,1}(x)=\int_{0}^{\tau_{1}^{\varepsilon,\delta}}\left[\phi_{z_{0}}^{\varepsilon,\delta},\cdots,\phi_{z'_{n-1}=z_{1}}^{\varepsilon,\delta}\right]_{0,t}(x)\phi_{z_{0}}(dx_{t}).
\]
But during $\left[0,\tau_{1}^{\varepsilon,\delta}\right]$ the path
$x$ never visits the interior of $H_{z_{1}}^{\varepsilon,\delta}$,
so the integral on the R.H.S. is zero and hence the extended signature
corresponding to $w'$ is zero. If $z'_{n}=z_{1}$,
\[
\left[\phi_{z_{0}}^{\varepsilon,\delta},\cdots,\phi_{z'_{n}}^{\varepsilon,\delta}\right]_{0,1}(x)=\int_{\tau_{1}^{\varepsilon,\delta}}^{1}\left[\phi_{z_{0}}^{\varepsilon,\delta},\cdots,\phi_{z'_{n-1}=z_{0}}^{\varepsilon,\delta}\right]_{0,t}(x)\phi_{z_{1}}^{\varepsilon,\delta}(dx_{t}).
\]
For $\tau_{1}^{\varepsilon,\delta}\leqslant t\leqslant1,$ we have
\begin{align*}
 & \left[\phi_{z_{0}}^{\varepsilon,\delta},\cdots,\phi_{z'_{n-1}=z_{0}}^{\varepsilon,\delta}\right]_{0,t}(x)\\
= & \left[\phi_{z_{0}}^{\varepsilon,\delta},\cdots,\phi_{z'_{n-1}=z_{0}}^{\varepsilon,\delta}\right]_{0,\tau_{1}^{\varepsilon,\delta}}(x)+\int_{\tau_{1}^{\varepsilon,\delta}}^{1}\left[\phi_{z_{0}}^{\varepsilon,\delta},\cdots,\phi_{z'_{n-2}=z_{1}}^{\varepsilon,\delta}\right]_{0,t}(x)\phi_{z_{0}}^{\varepsilon,\delta}(dx_{t})\\
= & \left[\phi_{z_{0}}^{\varepsilon,\delta},\cdots,\phi_{z'_{n-1}=z_{0}}^{\varepsilon,\delta}\right]_{0,\tau_{1}^{\varepsilon,\delta}}(x).
\end{align*}
But during $\left[0,\tau_{1}^{\varepsilon,\delta}\right]$ the path
$x$ does not visit the interior of $H_{z_{1}}^{\varepsilon,\delta}$
and the last term contains the one form $\phi_{z_{1}}^{\varepsilon,\delta},$
thus it is zero and $\left[\phi_{z_{0}}^{\varepsilon,\delta},\cdots,\phi_{z'_{n}}^{\varepsilon,\delta}\right]_{0,1}(x)=0$.
Therefore (2) again follows. 

Now assume that the claim is true for all non negative integer less
than $m,$ we are going to show that it is true for $m.$ Let $w=(z_{0},\cdots,z_{m})\in\mathcal{W}_{m}$
and $x\in E_{w}^{\varepsilon,\delta}.$

We first show (1). In fact,
\begin{align*}
\left[\phi_{z_{0}}^{\varepsilon,\delta},\cdots,\phi_{z_{m}}^{\varepsilon,\delta}\right]_{0,1}(x)= & \int_{0}^{\tau_{m}^{\varepsilon,\delta}}\left[\phi_{z_{0}}^{\varepsilon,\delta},\cdots,\phi_{z_{m-1}}^{\varepsilon,\delta}\right]_{0,t}(x)\phi_{z_{m}}^{\varepsilon,\delta}(dx_{t})\\
 & +\int_{\tau_{m}^{\varepsilon,\delta}}^{1}\left[\phi_{z_{0}}^{\varepsilon,\delta},\cdots,\phi_{z_{m-1}}^{\varepsilon,\delta}\right]_{0,t}(x)\phi_{z_{m}}^{\varepsilon,\delta}(dx_{t})\\
= & \int_{0}^{\tau_{m}^{\varepsilon,\delta}}\left[\phi_{z_{0}}^{\varepsilon,\delta},\cdots,\phi_{z_{m-1}}^{\varepsilon,\delta}\right]_{0,t}(x)\phi_{z_{m}}^{\varepsilon,\delta}(dx_{t})\\
 & +\left[\phi_{z_{0}}^{\varepsilon,\delta},\cdots,\phi_{z_{m-1}}^{\varepsilon,\delta}\right]_{0,\tau_{m}^{\varepsilon,\delta}}(x)\int_{\tau_{m}^{\varepsilon,\delta}}^{1}\phi_{z_{m}}^{\varepsilon,\delta}(dx_{t}),
\end{align*}
where the last equality comes from the fact that 
\[
\left[\phi_{z_{0}}^{\varepsilon,\delta},\cdots,\phi_{z_{m-1}}^{\varepsilon,\delta}\right]_{0,t}(x)=\left[\phi_{z_{0}}^{\varepsilon,\delta},\cdots,\phi_{z_{m-1}}^{\varepsilon,\delta}\right]_{0,\tau_{m}^{\varepsilon,\delta}}(x),\ \forall t\in\left[\tau_{m}^{\varepsilon,\delta},1\right],
\]
since $z_{m-1}\neq z_{m}$ and hence during $\left[\tau_{m}^{\varepsilon,\delta},1\right]$
the path does not visit the interior of $H_{z_{m-1}}^{\varepsilon,\delta}.$
Now we want to use the induction hypothesis (1) on the term $\left[\phi_{z_{0}}^{\varepsilon,\delta},\cdots,\phi_{z_{m-1}}^{\varepsilon,\delta}\right]_{0,\tau_{m}^{\varepsilon,\delta}}(x)$.
To this end, let $\widetilde{x}$ be a path in $W$ such that $\widetilde{x}=x$
on $\left[0,\tau_{m}^{\varepsilon,\delta}\right]$ and $\widetilde{x}$
stays inside the tunnel on $\left[\tau_{m}^{\varepsilon,\delta},1\right]$.
It follows that $\widetilde{x}\in E_{\widetilde{w}}^{\varepsilon,\delta}$
where $\widetilde{w}=(z_{0},\cdots,z_{m-1})\in\mathcal{W}_{m-1},$
and 
\[
\left[\phi_{z_{0}}^{\varepsilon,\delta},\cdots,\phi_{z_{m-1}}^{\varepsilon,\delta}\right]_{0,\tau_{m}^{\varepsilon,\delta}}(x)=\left[\phi_{z_{0}}^{\varepsilon,\delta},\cdots,\phi_{z_{m-1}}^{\varepsilon,\delta}\right]_{0,1}(\widetilde{x}).
\]
Therefore, by the induction hypothesis (1) and the definition of $\widetilde{x}$
we have
\begin{eqnarray*}
\left[\phi_{z_{0}}^{\varepsilon,\delta},\cdots,\phi_{z_{m-1}}^{\varepsilon,\delta}\right]_{0,1}(\widetilde{x}) & = & (\prod_{i=1}^{m-1}\int_{\tau_{i-1}^{\varepsilon,\delta}}^{\tau_{i}^{\varepsilon,\delta}}\phi_{z_{i-1}}^{\varepsilon,\delta}(d\widetilde{x}_{t}))(\int_{\tau_{m-1}^{\varepsilon,\delta}}^{1}\phi_{z_{m-1}}^{\varepsilon,\delta}(d\widetilde{x}_{t}))\\
 & = & (\prod_{i=1}^{m-1}\int_{\tau_{i-1}^{\varepsilon,\delta}}^{\tau_{i}^{\varepsilon,\delta}}\phi_{z_{i-1}}^{\varepsilon,\delta}(dx_{t}))(\int_{\tau_{m-1}^{\varepsilon,\delta}}^{\tau_{m}^{\varepsilon,\delta}}\phi_{z_{m-1}}^{\varepsilon,\delta}(dx_{t})).
\end{eqnarray*}
Consequently (1) will follow once we show that 
\[
\int_{0}^{\tau_{m}^{\varepsilon,\delta}}\left[\phi_{z_{0}}^{\varepsilon,\delta},\cdots,\phi_{z_{m-1}}^{\varepsilon,\delta}\right]_{0,t}(x)\phi_{z_{m}}^{\varepsilon,\delta}(dx_{t})=0.
\]
But this is an easy consequence of the fact that 
\[
\int_{0}^{\tau_{m}^{\varepsilon,\delta}}\left[\phi_{z_{0}}^{\varepsilon,\delta},\cdots,\phi_{z_{m-1}}^{\varepsilon,\delta}\right]_{0,t}(x)\phi_{z_{m}}^{\varepsilon,\delta}(dx_{t})=\left[\phi_{z_{0}}^{\varepsilon,\delta},\cdots,\phi_{z_{m}}^{\varepsilon,\delta}\right]_{0,1}(\widetilde{x})
\]
and the induction hypothesis (2). 

Now we show (2). Let $w'=(z_{0},z'_{1},\cdots,z'_{n})\in\mathcal{W}_{n}$
with $n>m.$ As before, the case when there exists some $0<k\leqslant n$
such that $z'_{k}\notin\{z_{0},\cdots,z_{m}\}$ is trivial. Otherwise,
write
\begin{equation}
\left[\phi_{z_{0}}^{\varepsilon,\delta},\cdots,\phi_{z_{n}'}^{\varepsilon,\delta}\right]_{0,1}(x)=\sum_{i}\int_{\tau_{i-1}^{\varepsilon,\delta}}^{\tau_{i}^{\varepsilon,\delta}}\left[\phi_{z_{0}}^{\varepsilon,\delta},\cdots,\phi_{z_{n-1}'}^{\varepsilon,\delta}\right]_{0,t}(x)\phi_{z'_{n}}^{\varepsilon,\delta}(dx_{t}),\label{prove (2) show each term=00003D0 in the sum}
\end{equation}
where the sum is over those $i\leqslant m+1$ such that $z_{i-1}=z'_{n}$.
Since $z'_{n-1}\neq z'_{n},$ for each such $i$ we have
\begin{align*}
 & \int_{\tau_{i-1}^{\varepsilon,\delta}}^{\tau_{i}^{\varepsilon,\delta}}\left[\phi_{z_{0}}^{\varepsilon,\delta},\cdots,\phi_{z_{n-1}'}^{\varepsilon,\delta}\right]_{0,t}(x)\phi_{z_{n}'}^{\varepsilon,\delta}(dx_{t})\\
= & \left[\phi_{z_{0}}^{\varepsilon,\delta},\cdots,\phi_{z_{n-1}'}^{\varepsilon,\delta}\right]_{0,\tau_{i-1}^{\varepsilon,\delta}}(x)\int_{\tau_{i-1}^{\varepsilon,\delta}}^{\tau_{i}^{\varepsilon,\delta}}\phi_{z'_{n}}^{\varepsilon,\delta}(dx_{t}).
\end{align*}
Define a new path $\widetilde{x}\in W$ such that $\widetilde{x}=x$
on $\left[0,\tau_{i-1}^{\varepsilon,\delta}\right]$ and $\widetilde{x}$
stays inside the tunnel on $\left[\tau_{i-1}^{\varepsilon,\delta},1\right].$
Then $\widetilde{x}\in E_{\widetilde{w}}^{\varepsilon,\delta}$ with
$\widetilde{w}=(z_{0},\cdots,z_{i-2}).$ Since during $\left[\tau_{i-1}^{\varepsilon,\delta},1\right]$
the path $\widetilde{x}$ does not visit the interior of $H_{z'_{n}}^{\varepsilon,\delta}$,
we have 
\[
\left[\phi_{z_{0}}^{\varepsilon,\delta},\cdots,\phi_{z_{n-1}'}^{\varepsilon,\delta}\right]_{0,\tau_{i-1}^{\varepsilon,\delta}}(x)=\left[\phi_{z_{0}}^{\varepsilon,\delta},\cdots,\phi_{z_{n-1}'}^{\varepsilon,\delta}\right]_{0,1}(\widetilde{x}).
\]
Now observe that $i-2<m\leqslant n-1,$ and so by the induction hypothesis
(2) we know that 
\[
\left[\phi_{z_{0}}^{\varepsilon,\delta},\cdots,\phi_{z_{n-1}'}^{\varepsilon,\delta}\right]_{0,1}(\widetilde{x})=0.
\]
Therefore, each term in the R.H.S. is zero and (2) follows.

Finally we show (3). Let $w'=(z_{0},z'_{1},\cdots,z'_{m})\in\mathcal{W}_{m}$
with $w'\neq w.$ If $z'_{m}=z_{m}$, then 
\begin{align*}
 & \left[\phi_{z_{0}}^{\varepsilon,\delta},\cdots,\phi_{z'_{m}}^{\varepsilon,\delta}\right]_{0,1}(x)\\
= & \left[\phi_{z_{0}}^{\varepsilon,\delta},\cdots,\phi_{z'_{m}}^{\varepsilon,\delta}\right]_{0,\tau_{m}^{\varepsilon,\delta}}(x)+\left[\phi_{z_{0}}^{\varepsilon,\delta},\cdots,\phi_{z_{m-1}'}^{\varepsilon,\delta}\right]_{0,\tau_{m}^{\varepsilon,\delta}}(x)\int_{\tau_{m}^{\varepsilon,\delta}}^{1}\phi_{z_{m}}^{\varepsilon,\delta}(dx_{t}).
\end{align*}
Define $\widetilde{x}\in W$ by $\widetilde{x}=x$ on $\left[0,\tau_{m}^{\varepsilon,\delta}\right]$
and staying inside the tunnel on $\left[\tau_{m}^{\varepsilon,\delta},1\right]$.
It follows from the induction hypothesis (2) that 
\[
\left[\phi_{z_{0}}^{\varepsilon,\delta},\cdots,\phi_{z'_{m}}^{\varepsilon,\delta}\right]_{0,\tau_{m}^{\varepsilon,\delta}}(x)=\left[\phi_{z_{0}}^{\varepsilon,\delta},\cdots,\phi_{z'_{m}}^{\varepsilon,\delta}\right]_{0,1}(\widetilde{x})=0.
\]
Moreover, in this case we know that $(z_{0},\cdots,z'_{m-1})\neq(z_{0},\cdots,z_{m-1})$.
Therefore, by induction hypothesis (3) we have 
\[
\left[\phi_{z_{0}}^{\varepsilon,\delta},\cdots,\phi_{z_{m-1}'}^{\varepsilon,\delta}\right]_{0,\tau_{m}^{\varepsilon,\delta}}(x)=\left[\phi_{z_{0}}^{\varepsilon,\delta},\cdots,\phi_{z_{m-1}'}^{\varepsilon,\delta}\right]_{0,1}(\widetilde{x})=0.
\]
Consequently (3) follows. For the case $z'_{m}\neq z_{m}$ and there
exists some $i\leqslant m+1$ with $z_{i-1}=z'_{m}$ (otherwise it
is trivial), we know that $i$ must be strictly less than $m-1.$
By writing $\left[\phi_{z_{0}}^{\varepsilon,\delta},\cdots,\phi_{z'_{m}}^{\varepsilon,\delta}\right]_{0,1}(x)$
as a sum of the form (\ref{prove (2) show each term=00003D0 in the sum}),
the result (3) will follow easily from the induction hypothesis (2)
by a similar argument.

Now the proof is complete.
\end{proof}

Define a map $M^{\varepsilon,\delta}:\ W\rightarrow\mathbb{Z}_{+}$
by sending a path $x\in W$ to 
\[
\sup\left\{ m\geqslant0:\ \exists w=(z_{0},z_{1},\cdots,z_{m})\in\mathcal{W}_{m}\ \mathrm{s.t.}\ \left[\phi_{z_{0}}^{\varepsilon,\delta},\phi_{z_{1}}^{\varepsilon,\delta},\cdots,\phi_{z_{m}}^{\varepsilon,\delta}\right]_{0,1}(x)\neq0\right\} .
\]
Note that by Lemma \ref{recovering the ordered sequence of boxes},
$M^{\varepsilon,\delta}\leqslant N^{\varepsilon,\delta}$ for $\mathbb{P}$-almost
surely. Moreover, we are able to prove the following recovery result.
\begin{prop}
\label{unique nonzero}For each $x\in W$ outside a $\mathbb{P}$-null
set, there exists a unique word $w=(z_{0},\cdots,z_{M^{\varepsilon,\delta}(x)})\in\mathcal{W}_{M^{\varepsilon,\delta}(x)}$
such that 
\[
\left[\phi_{z_{0}}^{\varepsilon,\delta},\cdots,\phi_{z_{M^{\varepsilon,\delta}(x)}}^{\varepsilon,\delta}\right]_{0,1}(x)\neq0.
\]
This word is exactly given by $M^{\varepsilon,\delta}(x)=N^{\varepsilon,\delta}(x),$
and 
\[
z_{i}=\boldsymbol{m}_{i}^{\varepsilon,\delta}(x),\ i=0,\cdots,M^{\varepsilon,\delta}(x).
\]
\end{prop}
\begin{proof}
Let $\mathcal{N}^{\varepsilon,\delta}$ be the set 
\[
\bigcup_{m=0}^{\infty}\bigcup_{w=(z_{0},\cdots,z_{m})\in\mathcal{W}_{m}}\bigcup_{i=0}^{m}\bigcup_{\substack{0\leqslant r_{1}<r_{2}\leqslant1\\
r_{1},r_{2}\in\mathbb{Q}
}
}\left(\left\{ x\in W:\ \int_{r_{1}}^{r_{2}}\phi_{z_{i}}^{\varepsilon,\delta}(dx_{u})=0\right\} \bigcap A_{r_{1},r_{2}}^{z_{i},\varepsilon,\delta}\right),
\]
where $A_{r_{1},r_{2}}^{z_{i},\varepsilon,\delta}$ is the set defined
in (\ref{travel through the interior}) associated with the cube $H_{z_{i}}^{\varepsilon,\delta}$
and the differential one form $\phi_{z_{i}}^{\varepsilon,\delta}.$
It follows from Assumption (C) that $\mathcal{N}^{\varepsilon,\delta}$
is a $\mathbb{P}$-null set. 

For any $x\in(\mathcal{N}^{\varepsilon,\delta})^{c},$ let $w=(z_{0},\cdots,z_{m})$
be the word in $\mathcal{W}_{m}$ with $m=N^{\varepsilon,\delta}$
and $z_{i}=\boldsymbol{m}_{i}^{\varepsilon,\delta}$, for $i=0,\ldots,m$,
so $x\in E_{w}^{\varepsilon,\delta}$. 

By (\ref{breaking into product}) in Lemma \ref{recovering the ordered sequence of boxes},
if $\left[\phi_{z_{0}}^{\varepsilon,\delta},\cdots,\phi_{z_{m}}^{\varepsilon,\delta}\right]_{0,1}(x)=0$,
then there exists some $i=1,\cdots,m+1$ such that $\int_{\tau_{i-1}^{\varepsilon,\delta}}^{\tau_{i}^{\varepsilon,\delta}}\phi_{z_{i-1}}^{\varepsilon,\delta}(dx_{t})=0.$
By the definition of $\tau_{k}^{\varepsilon,\delta}$ and continuity,
we can find some rational numbers $r_{1}<\tau_{i-1}^{\varepsilon,\delta}$
and $r_{2}<\tau_{i}^{\varepsilon,\delta}$ (if $m=0$ take $r_{1}=0$
and $r_{2}=1$; otherwise if $i=1,$ take $r_{1}=0$ and if $i=m+1,$
take $r_{2}=1$) such that there exists some $u\in(r_{1},r_{2})$
with $x_{u}\in H_{z_{i-1}}^{\varepsilon,\delta}$ and 
\[
\int_{r_{1}}^{r_{2}}\phi_{z_{i-1}}^{\varepsilon,\delta}(dx_{t})=\int_{\tau_{i-1}^{\varepsilon,\delta}}^{\tau_{i}^{\varepsilon,\delta}}\phi_{z_{i-1}}^{\varepsilon,\delta}(dx_{t})=0.
\]
This implies that $x\in\mathcal{N}^{\varepsilon,\delta},$ which is
a contradiction. Therefore, we have $\left[\phi_{z_{0}}^{\varepsilon,\delta},\cdots,\phi_{z_{m}}^{\varepsilon,\delta}\right]_{0,1}(x)\neq0.$ 

By the second and third part of Lemma \ref{recovering the ordered sequence of boxes},
we know that $M^{\varepsilon,\delta}(x)=m$ and $ $$w$ is the unique
word in $\mathcal{W}_{m}$ such that the corresponding extended signature
of $x$ is nonzero.
\end{proof}

Together with the result in Section 4, proposition \ref{unique nonzero}
tells us that outside a $\mathbb{P}$-null set, given the signature
of a path $x$ we can recover the sequence of open cubes $H_{z}^{\varepsilon,\delta}$
which $x$ has visited in order.

\subsection{An Approximation Result}

Now we are going to construct a polygonal approximation of $ $a path
based on the ordered sequence of open cubes visited by the path and
the corresponding visit times. With probability one, such polygonal
approximation converges to the original path under the uniform topology.
This result is crucial for the recovery of a path up to reparametrization
from its signature.

Let $x\in W$ and define the word $w=(z_{0},\cdots,z_{m})\in\mathcal{W}_{m}$
by $m=N^{\varepsilon,\delta}$ and $z_{i}=\boldsymbol{m}_{i}^{\varepsilon,\delta}$
for $i=0,\cdots,m$. Construct a polygonal path $x^{\varepsilon,\delta}$
as follows. If $m=0,$ let $x_{t}^{\varepsilon,\delta}=0$ for $t\in[0,1]$;
otherwise for $1\leqslant k\leqslant m,$ define
\[
x_{t}^{\varepsilon,\delta}=\frac{\tau_{k}^{\varepsilon,\delta}-t}{\tau_{k}^{\varepsilon,\delta}-\tau_{k-1}^{\varepsilon,\delta}}\varepsilon z_{k-1}+\frac{t-\tau_{k-1}^{\varepsilon,\delta}}{\tau_{k}^{\varepsilon,\delta}-\tau_{k-1}^{\varepsilon,\delta}}\varepsilon z_{k},\ t\in\left[\tau_{k-1}^{\varepsilon,\delta},\tau_{k}^{\varepsilon,\delta}\right],
\]
and 
\[
x_{t}^{\varepsilon,\delta}=\varepsilon z_{m},\ t\in\left[\tau_{m}^{\varepsilon,\delta},1\right].
\]
The approximation scheme is illustrated by Figure 1.

\begin{figure}
\begin{center}

\includegraphics[scale=0.45]{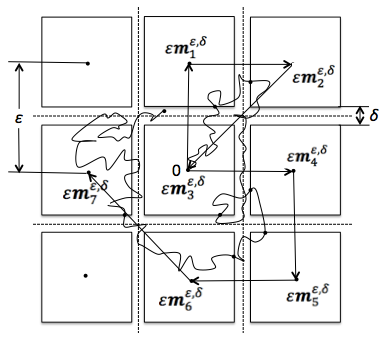}

\caption{This figure illustrates the corresponding approximation scheme. The
dotted lines represent the degenerate tunnels. According to Assumption
(B) on the process, the probability that a path stays in these tunnels
for a positive time period is zero, a crucial fact used in the proof
of Proposition \ref{prop:approximation}.}

\end{center}
\end{figure}

Now we have the following approximation result.
\begin{prop}
\label{prop:approximation}For each $n\geqslant1$ and $\varepsilon_{n}=1/n$,
there exists $\delta_{n}>0,$ such that for $\mathbb{P}$-almost surely,
\begin{equation}
\lim_{n\rightarrow\infty}\sup_{0\leqslant t\leqslant1}\left|x_{t}^{\varepsilon_{n},\delta_{n}}-x_{t}\right|=0.\label{uniform convergence}
\end{equation}
\end{prop}
\begin{proof}
For each $\varepsilon,\delta,$ let 
\[
T^{\varepsilon,\delta}=\mathbb{R}^{d}\backslash\bigcup_{z\in\mathbb{Z}^{d}}H_{z}^{\varepsilon,\delta}
\]
be the set of closed tunnels, and define 
\[
A^{\varepsilon,\delta}=\left\{ x\in W:\ \exists\left[s,t\right]\subset x^{-1}\left(T^{\varepsilon,\delta}\right),\left|x_{t}-x_{s}\right|\geqslant\varepsilon\right\} .
\]

We first show that for any fixed $\varepsilon>0,$
\begin{align}
\bigcap_{\delta>0}A^{\varepsilon,\delta}\subset & \left\{ x\in W:\ \exists1\leqslant i\leqslant d,k\in\mathbb{Z},q\in\mathbb{Q}\bigcap(0,1)\ \mathrm{s.t.}\; x_{q}^{i}=\frac{2k-1}{2}\varepsilon\right\} .\label{eq:narrow tunnel}
\end{align}
Let $x\in\bigcap_{\delta>0}A^{\varepsilon,\delta}$, and $\delta_{n}$
be a sequence such that $\delta_{n}\downarrow0$. Then for each $n\geqslant1$,
there exists $0\leqslant s_{n}<t_{n}\leqslant1$ such that $\left[s_{n},t_{n}\right]\subset x^{-1}\left(T^{\varepsilon,\delta_{n}}\right)\;\mbox{and }\left|x_{t_{n}}-x_{s_{n}}\right|\geqslant\varepsilon.$
By compactness we can find a subsequence $(s_{n_{l}},t_{n_{l}})$
of $(s_{n},t_{n})$ such that $(s_{n_{l}},t_{n_{l}})$ converges to
some $\left(s,t\right)$. The condition $\left|x_{t_{n_{l}}}-x_{s_{n_{l}}}\right|\geqslant\varepsilon$
then implies that $s<t$. Therefore, for fixed $u,v$ with $s<u<v<t$,
there exists some $N\in\mathbb{N}$ such that $\left[u,v\right]\subset\bigcap_{l\geqslant N}[s_{n_{l}},t_{n_{l}}]$,
and hence 
\begin{eqnarray*}
\left[u,v\right] & \subset & \bigcap_{l\geqslant N}x^{-1}\left(T^{\varepsilon,\delta_{n_{l}}}\right)\\
 & = & x^{-1}\left(\bigcup_{k\in\mathbb{Z}}\bigcup_{1\leqslant i\leqslant d}\mathbb{R}^{i-1}\times\left\{ \frac{2k-1}{2}\varepsilon\right\} \times\mathbb{R}^{d-i}\right).
\end{eqnarray*}
In particular, this implies (\ref{eq:narrow tunnel}) and by Assumption
(B) we have $\mathbb{P}\left(\bigcap_{\delta>0}A^{\varepsilon,\delta}\right)=0$. 

Now we are going to show that for each $\varepsilon,\delta,$ 
\begin{equation}
\left\{ x\in W:\ \sup_{0\leqslant u\leqslant1}\left|x_{u}^{\varepsilon,\delta}-x_{u}\right|\geqslant11\sqrt{d}\varepsilon\right\} \subset A^{\varepsilon,\delta}.\label{travelling through tunnels}
\end{equation}
To see this, first notice that if $x$ belongs to the left hand side
of (\ref{travelling through tunnels}), then either

(1) there exists some $u\in[\tau_{k-1}^{\varepsilon,\delta},\tau_{k}^{\varepsilon,\delta}]$
for some $1\leqslant k\leqslant N^{\varepsilon,\delta},$ such that
$\left|x_{u}^{\varepsilon,\delta}-x_{u}\right|\geqslant11\sqrt{d}\varepsilon$;
or

(2) there exists some $u\in[\tau_{N^{\varepsilon,\delta}}^{\varepsilon,\delta},1],$
such that $\left|x_{u}-\varepsilon\boldsymbol{m}_{N^{\varepsilon,\delta}}^{\varepsilon,\delta}\right|\geqslant11\sqrt{d}\varepsilon.$

In the first case, we know that $x$ does not visit any cube other
than $H_{\boldsymbol{m}_{k-1}^{\varepsilon,\delta}}^{\varepsilon,\delta}$
during $(\tau_{k-1}^{\varepsilon,\delta},\tau_{k}^{\varepsilon,\delta})$.
If the distance between the cubes $H_{\boldsymbol{m}_{k}^{\varepsilon,\delta}}^{\varepsilon,\delta}$
and $H_{\boldsymbol{m}_{k-1}^{\varepsilon,\delta}}^{\varepsilon,\delta}$
is at least $3\sqrt{d}\varepsilon$, by continuity there exist $\tau_{k-1}^{\varepsilon,\delta}<s<t<\tau_{k}^{\varepsilon,\delta}$,
such that 
\[
\left|x_{s}-x_{\tau_{k-1}^{\varepsilon,\delta}}\right|=\sqrt{d}\varepsilon,\;\left|x_{t}-x_{\tau_{k-1}^{\varepsilon,\delta}}\right|=2\sqrt{d}\varepsilon,
\]
and $[s,t]\subset x^{-1}\left(T^{\varepsilon,\delta}\right).$ Moreover,
by the triangle inequality we have $|x_{t}-x_{s}|\geqslant\varepsilon$.
Therefore, $x\in A^{\varepsilon,\delta}$. If the distance between
$H_{\boldsymbol{m}_{k}^{\varepsilon,\delta}}^{\varepsilon,\delta}$
and $H_{\boldsymbol{m}_{k-1}^{\varepsilon,\delta}}^{\varepsilon,\delta}$
is strictly less than $3\sqrt{d}\varepsilon$, we know that $\left|x_{u}^{\varepsilon,\delta}-\varepsilon\boldsymbol{m}_{k-1}^{\varepsilon,\delta}\right|\leqslant4\sqrt{d}\varepsilon$
for all $u\in\left(\tau_{k-1}^{\varepsilon,\delta},\tau_{k}^{\varepsilon,\delta}\right)$.
Since $\sup_{0\leqslant u\leqslant1}\left|x_{u}^{\varepsilon,\delta}-x_{u}\right|\geqslant11\sqrt{d}\varepsilon,$
there exists $u\in\left(\tau_{k-1}^{\varepsilon,\delta},\tau_{k}^{\varepsilon,\delta}\right)$
such that 
\[
\left|x_{u}-\varepsilon\boldsymbol{m}_{k-1}^{\varepsilon,\delta}\right|,\ \left|x_{u}-\varepsilon\boldsymbol{m}_{k}^{\varepsilon,\delta}\right|\geqslant7\sqrt{d}\varepsilon.
\]
It follows again from continuity that there exist $\tau_{k-1}^{\varepsilon,\delta}<s<t<\tau_{k}^{\varepsilon,\delta}$
such that 
\[
\left|x_{s}-\varepsilon\boldsymbol{m}_{k-1}^{\varepsilon,\delta}\right|=5\sqrt{d}\varepsilon,\;\left|x_{t}-\varepsilon\boldsymbol{m}_{k-1}^{\varepsilon,\delta}\right|=6\sqrt{d}\varepsilon,
\]
and $[s,t]\subset x^{-1}\left(T^{\varepsilon,\delta}\right).$ Therefore,
$\left|x_{s}-x_{t}\right|\geqslant\varepsilon$ and we have $x\in A^{\varepsilon,\delta}.$ 

In the second case, there exist $\tau_{N^{\varepsilon,\delta}}^{\varepsilon,\delta}<s<t\leqslant1$
such that 
\[
\left|x_{s}-\varepsilon\boldsymbol{m}_{N^{\varepsilon,\delta}}^{\varepsilon,\delta}\right|=\sqrt{d}\varepsilon,\;\left|x_{t}-\varepsilon\boldsymbol{m}_{N^{\varepsilon,\delta}}^{\varepsilon,\delta}\right|=2\sqrt{d}\varepsilon,
\]
and $[s,t]\subset x^{-1}(T^{\varepsilon,\delta}).$ Again we have
$|x_{t}-x_{s}|\geqslant\varepsilon$ and hence $x\in A^{\varepsilon,\delta}.$ 

Now for $\varepsilon_{n}=1/n$, if we choose $\delta_{n}$ small enough
such that $\mathbb{P}\left(A^{\varepsilon_{n},\delta_{n}}\right)\leqslant\varepsilon_{n}^{2}$,
we have 
\begin{eqnarray*}
\sum_{n=1}^{\infty}\mathbb{P}\left(\left\{ x\in W:\ \sup_{0\leqslant u\leqslant1}|x_{u}^{\varepsilon_{n},\delta_{n}}-x_{u}|\geqslant11\sqrt{d}\varepsilon_{n}\right\} \right) & \leqslant & \sum_{n=1}^{\infty}\mathbb{P}\left(A^{\varepsilon_{n},\delta(\varepsilon_{n})}\right)\\
 & < & \infty,
\end{eqnarray*}
It follows from the Borel-Cantelli lemma that 
\[
\mathbb{P}\left(\limsup_{n\rightarrow\infty}\left\{ x\in W:\ \sup_{0\leqslant u\leqslant1}|x_{u}^{\varepsilon_{n},\delta_{n}}-x_{u}|\geqslant11\sqrt{d}\varepsilon_{n}\right\} \right)=0,
\]
and hence the uniform convergence (\ref{uniform convergence}) holds
for $\mathbb{P}$-almost surely.
\end{proof}

\begin{rem}
From the previous proof, it is not hard to see that the result of
Proposition \ref{prop:approximation} holds for all continuous stochastic
processes starting at the origin whose law satisfies Assumption (B).
\end{rem}

From now on, we will always assume that $\varepsilon_{n}=1/n,$ and
take $\delta_{n}$ as in the previous proof.

\subsection{A Variant of the Fréchet Distance on Path Space}

Now we are coming to the last step of the proof of Theorem \ref{main thm general framework}. 

Under Assumption (A), (B), (C), what we've obtained so far is that
there exists some $\mathbb{P}$-null set $\mathcal{N},$ such that
for any path $x\in\mathcal{N}^{c},$ the signature $S(x)_{0,1}$ is
well-defined, and for each $n\geqslant1,$ we can recover the ordered
sequence of open cubes $H_{z}^{\varepsilon_{n},\delta_{n}}$ visited
by $x$ from its signature. Moreover, the polygonal approximation
$x^{\varepsilon_{n},\delta_{n}}$ constructed before converges to
$x$ uniformly. 

By possibly enlarging the $\mathbb{P}$-null set $\mathcal{N}$ (still
a $\mathbb{P}$-null set), we are going to show that for any two paths
$x,x'\in\mathcal{N}^{c},$ if $S(x)_{0,1}=S(x')_{0,1}$, then $x$
and $x'$ defer by a reparametrization $\sigma\in\mathcal{R}$ in
the sense of Definition \ref{reparametrization}.

Now we introduce an equivalence relation ``$\thicksim$'' on $W$
by
\[
x\thicksim x'\iff\left(x_{t}\right)_{0\leqslant t\leqslant1}=\left(x'_{\sigma(t)}\right)_{0\leqslant t\leqslant1},\ \mbox{for some \ensuremath{\sigma\in\mathcal{R}}.}
\]
Let $W/_{\thicksim}$ be the quotient space consisting of $\thicksim$-equivalence
classes. For any $[x],[x']\in W/_{\sim},$ define 
\begin{equation}
d\left([x],[x']\right)=\inf_{\sigma\in\mathcal{R}}\sup_{t\in[0,1]}\left|x_{t}-x'_{\sigma(t)}\right|.\label{metric}
\end{equation}
If we only assume that $\sigma$ is non-decreasing, the function $d(\cdot,\cdot)$
is usually know as the Fréchet distance. It was originally introduced
by Fr$\acute{\mathrm{e}}$chet to study the shape of geometric spaces.
Here we emphasize that $\sigma$ is strictly increasing.

It is easy to see that $d(\cdot,\cdot)$ does not depend on the choice
of representatives in the corresponding equivalence classes, and $d(\cdot,\cdot)$
is nonnegative and symmetric. Moreover, $d(\cdot,\cdot)$ satisfies
the triangle inequality. In fact, for any $x,x',x''\in W$ and $\sigma,\theta\in\mathcal{R},$
we have 
\[
\sup_{t\in[0,1]}\left|x_{t}-x''_{\sigma(t)}\right|\leqslant\sup_{t\in[0,1]}\left|x_{t}-x'_{\theta(t)}\right|+\sup_{t\in[0,1]}\left|x'_{\theta(t)}-x''_{\sigma(t)}\right|.
\]
It follows that 
\begin{eqnarray*}
d\left([x],[x'']\right) & = & \inf_{\sigma\in\mathcal{R}}\sup_{t\in[0,1]}\left|x_{t}-x''_{\sigma(t)}\right|\\
 & \leqslant & \sup_{t\in[0,1]}\left|x_{t}-x'_{\theta(t)}\right|+\inf_{\sigma\in\mathcal{R}}\sup_{t\in[0,1]}\left|x'_{\theta(t)}-x''_{\sigma(t)}\right|\\
 & = & \sup_{t\in[0,1]}\left|x_{t}-x'_{\theta(t)}\right|+d\left([x'],[x'']\right).
\end{eqnarray*}
By taking infimum over $\theta\in\mathcal{R}$, we obtain the triangle
inequality.

It should be pointed out that unlike the Fréchet distance, $d(\cdot,\cdot)$
is not a metric on $W/_{\thicksim}.$ For example, consider the case
of $d=1.$ Let $x_{t}=t,\ t\in[0,1]$, and 
\[
x'_{t}=\begin{cases}
2t, & t\in[0,\frac{1}{2}];\\
1, & t\in[\frac{1}{2},1].
\end{cases}
\]
Then it is easy to see that $d\left([x],[x']\right)=0$, but obviously
$x'$ is not a reparametrization of $x$ in the sense of Definition
\ref{reparametrization}. However, if we exclude paths with certain
degeneracy, then on the corresponding quotient space $d(\cdot,\cdot)$
is indeed a metric. 

Let $D$ be the set of paths $x\in W$ such that there exist some
$0\leqslant s<t\leqslant1$ with 
\[
x_{u}=x_{s},\ \forall u\in[s,t].
\]
We first make an important remark that under Assumption (C), $D$
is a $\mathbb{P}$-null set. To see this, let $\{H_{n}\}_{n\geqslant1}$
be a covering of $\mathbb{R}^{d}$ consisting of open cubes, and for
each $n$ let $\phi_{n}$ be the differential one form associated
with $H_{n}$ according to Assumption $(C).$ It follows that 
\[
D\subset\bigcup_{r_{1},r_{2}\in\mathbb{Q}\cap[0,1]}\bigcup_{n\geqslant1}\left(\left\{ x\in W:\ \int_{r_{1}}^{r_{2}}\phi_{n}(dx_{u})=0\right\} \bigcap A_{r_{1},r_{2}}^{H_{n}}\right\} .
\]
Therefore, by Assumption (C) we know that $\mathbb{P}(D)=0.$ 

Now we have the following result.
\begin{prop}
\label{d is a metric}Define the equivalence relation ``$\thicksim$''
on $W_{0}=D^{c}\subset W$ as before, and let $W_{0}/_{\thicksim}$
be the corresponding quotient space. Then $d(\cdot,\cdot)$, defined
in the same way as in (\ref{metric}), is a metric on $W_{0}/_{\thicksim}.$\end{prop}
\begin{proof}
It suffices to show that, for any $x,x'\in W_{0},$ if 
\begin{equation}
\inf_{\sigma\in\mathcal{R}}\sup_{t\in[0,1]}\left|x_{t}-x'_{\sigma(t)}\right|=0,\label{distance is zero}
\end{equation}
then 
\begin{equation}
x_{t}=x'_{\sigma(t)},\ \forall t\in[0,1],\label{equal up to reparametrization}
\end{equation}
for some $\sigma\in\mathcal{R}.$

In fact, by (\ref{distance is zero}), for any $n\geqslant1,$ there
exists $\sigma_{n}\in\mathcal{R},$ such that
\begin{equation}
\left|x_{t}-x'_{\sigma_{n}(t)}\right|\leqslant\frac{1}{n},\ \forall t\in[0,1].\label{uniform control in metric}
\end{equation}
It follows from compactness, denseness, and a standard diagonal selection
argument that we can find a subsequence $\{\sigma_{n_{k}}\}$ such
that for any $r\in\mathbb{Q}\bigcap[0,1],$
\[
\lim_{k\rightarrow\infty}\sigma_{n_{k}}(r)=:\widetilde{\sigma}(r)
\]
exists.

Now define $\sigma:\ [0,1]\rightarrow[0,1]$ by 
\[
\sigma(t)=\begin{cases}
\inf\left\{ \widetilde{\sigma}(r):\ r>t,r\in\mathbb{Q}\bigcap[0,1]\right\} , & 0\leqslant t<1;\\
1, & t=1.
\end{cases}
\]
We want to show that $\sigma\in\mathcal{R},$ and it satisfies (\ref{equal up to reparametrization}).

(1) It is easy to see that $\sigma$ is increasing. Let $0\leqslant t<1.$
For any $\varepsilon>0,$ there exists some $r>t,r\in\mathbb{Q}\bigcap[0,1],$
such that 
\[
\sigma(t)\leqslant\widetilde{\sigma}(r)<\sigma(t)+\varepsilon.
\]
Therefore, for any $t'\in(t,r),$ if we take some $r'\in\mathbb{Q}\bigcap[0,1]$
with $t'<r'<r,$ then
\[
\sigma(t)\leqslant\sigma(t')\leqslant\widetilde{\sigma}(r')\leqslant\widetilde{\sigma}(r)<\sigma(t)+\varepsilon.
\]
It follows that $\sigma$ is right continuous.

(2) $\sigma$ is also left continuous. 

In fact, assume on the contrary that for some $0<t\leqslant1,$ $\sigma(t-)\neq\sigma(t).$
Fix any $\sigma(t-)<s<\sigma(t),$ and define for $k\geqslant1$,
$t_{n_{k}}=\sigma_{n_{k}}^{-1}(s)$. It follows that for any $r>t,r\in\mathbb{Q}\bigcap[0,1],$
\[
s<\sigma(t)\leqslant\widetilde{\sigma}(r).
\]
Since $\lim_{k\rightarrow\infty}\sigma_{n_{k}}(r)=\widetilde{\sigma}(r),$
we know that when $k$ is large enough, $s<\sigma_{n_{k}}(r),$ which
is equivalent to $t_{n_{k}}<r$ for $k$ large enough. Therefore,
we have $\limsup_{k\rightarrow\infty}t_{n_{k}}\leqslant r.$ But this
is true for all $r>t,r\in\mathbb{Q}\bigcap[0,1],$ which implies that
$\limsup_{k\rightarrow\infty}t_{n_{k}}\leqslant t.$ On the other
hand, for any $r<t,r\in\mathbb{Q}\bigcap[0,1],$ we have 
\[
\widetilde{\sigma}(r)\leqslant\sigma(r)\leqslant\sigma(t-)<s,
\]
A similar argument yields that $\liminf_{k\rightarrow\infty}t_{n_{k}}\geqslant t.$
Therefore, $\lim_{k\rightarrow\infty}t_{n_{k}}$ exists and is equal
to $t.$ Now from (\ref{uniform control in metric}) we know that
\[
|x_{t_{n_{k}}}-x'_{s}|\leqslant\frac{1}{n_{k}},\ \forall k\geqslant1,
\]
and hence $x_{t}=x'_{s}.$ But this is true for all $\sigma(t-)<s<\sigma(t),$
which contradicts the fact that $x'\in W_{0}.$ Therefore, $\sigma$
is left continuous. A similar argument also shows that $\sigma(0)=0.$

(3) For any $r\in\mathbb{Q}\bigcap[0,1],$ $\sigma(r)=\widetilde{\sigma}(r).$ 

In fact, it is obvious that $\sigma(r)\geqslant\widetilde{\sigma}(r).$
On the other hand, for any $t<r$ we have $\sigma(t)\leqslant\widetilde{\sigma}(r),$
and by the left continuity of $\sigma$ we have $\sigma(r)\leqslant\widetilde{\sigma}(r).$

(4) $\sigma$ is strictly increasing. 

In fact, if for some $0\leqslant s<t\leqslant1,$ $\sigma(s)=\sigma(t),$
then $\sigma$ remains constant over $[s,t].$ In particular, for
any $r\in\mathbb{Q}\bigcap[s,t],$ from (\ref{uniform control in metric})
and the previous step we have
\[
x_{r}=x'_{\widetilde{\sigma}(r)}=x'_{\sigma(r)}=x'_{\sigma(s)},
\]
which implies that $x$ is constant over $[s,t],$ contradicting the
fact that $x\in W_{0}.$

Now it is obvious that $\sigma\in\mathcal{R},$ and ( \ref{equal up to reparametrization}
) follows. 
\end{proof}

From now on, we shall include $D$ to the $\mathbb{P}$-null set $\mathcal{N}.$

Now we are in a position to complete the proof of Theorem \ref{main thm general framework}. 

Assume that $x,x'\in\mathcal{N}^{c}$ and $S(x)_{0,1}=S(x')_{0,1}$.
For each $n\geqslant1$, let $\left(\phi_{z_{0}}^{\varepsilon_{n},\delta_{n}},\cdots,\phi_{z_{m}}^{\varepsilon_{n},\delta_{n}}\right)$
($\left(\phi_{z_{0}}^{\varepsilon_{n},\delta_{n}},\cdots,\phi_{z'_{m'}}^{\varepsilon_{n},\delta_{n}}\right)$,
respectively) be the unique maximal sequence of differential one forms
along which the extended signature of $x$ ($x'$, respectively) is
nonzero. It follows from Theorem \ref{signature determines extended signatures}
that $m=m'$ and $z_{i}=z'_{i}$ for $i=0,\ldots,m$. Moreover, by
Proposition \ref{unique nonzero} we know that 
\[
N^{\varepsilon_{n},\delta_{n}}(x)=N^{\varepsilon_{n},\delta_{n}}(x')=m,
\]
and 
\[
\boldsymbol{m}_{i}^{\varepsilon_{n},\delta_{n}}(x)=\boldsymbol{m}_{i}^{\varepsilon_{n},\delta_{n}}(x')=z_{i},\ \forall i=0,\cdots,m.
\]
It follows that in the quotient space $W/_{\sim},$ $[x^{\varepsilon_{n},\delta_{n}}]=[(x')^{\varepsilon_{n},\delta_{n}}]$,
where $x^{\varepsilon_{n},\delta_{n}}$ and $(x')^{\varepsilon_{n},\delta_{n}}$
are the polygonal approximations of $x$ and $x'$ respectively. On
the other hand, by Proposition \ref{prop:approximation} we know that
\[
x^{\varepsilon_{n},\delta_{n}}\rightarrow x,\ (x')^{\varepsilon_{n},\delta_{n}}\rightarrow x',
\]
under the uniform topology as $n\rightarrow\infty$. Therefore, by
the triangle inequality of the distance function $d(\cdot,\cdot)$
we have $d([x],[x'])=0.$ Since $D\subset\mathcal{N}$, it follows
from Proposition \ref{d is a metric} that there exists $\sigma\in\mathcal{R},$
such that (\ref{equal up to reparametrization}) holds. 

Now the proof of Theorem\ref{main thm general framework} is complete.

\section{A Fundamental Example: Gaussian Processes}

As we remarked before, Assumption (A) and (B) are natural for a large
class of stochastic processes. However, Assumption (C) is in general
difficult to verify. In this section, as a fundamental example of
Theorem \ref{main thm general framework}, we are going to show that
Assumption (A), (B), (C) hold for a class of Gaussian processes including
fractional Brownian motion with Hurst parameter $H>1/4$, the Ornstein-Uhlenbeck
process and the Brownian bridge. The main idea of verifying Assumption
(C) for Gaussian processes is to apply local regularity results for
Gaussian functionals from the Malliavin calculus, based on pathwise
integration by parts which is possible due to the regularity of sample
paths and Cameron-Martin paths.

The class of Gaussian processes we shall study in this section is
specified in the following. 

Let $\mathbb{P}$ be the law of a centered, nondegenerate, continuous
Gaussian process over $[0,1]$ starting at the origin with i.i.d components.
We assume that $\mathbb{P}$ satisfies the following conditions: there
exists $H\in\left(\frac{1}{4},1\right)$ such that

(G1) for all $\rho\in\left(\frac{1}{2H}\vee1,2\right]$, the $\rho$-variation
of the covariance function (see \cite{FV10}, Definition 5. 50) of
each component of $X$ is controlled by a 2D Hölder-dominated control
(see \cite{FV10}, Definition 5.51);

(G2) there exists $\delta>0$ and $c_{\delta}>0$, such that for all
$0\leqslant s<t\leqslant1$ with $\left|t-s\right|\leqslant\delta$,
we have
\[
\mathbb{E}\left[(X_{t}-X_{s})^{2}\right]\geqslant c_{\delta}(t-s)^{2H};
\]

(G3) the Cameron-Martin space $\mathcal{H}$ associated with $\mathbb{P}$
satisfies the property that 
\[
C_{0}^{1+H-}([0,1];\mathbb{R}^{d})\subset\mathcal{H}\subset C_{0}^{q-var}([0,1];\mathbb{R}^{d}),\ \forall q>\left(H+\frac{1}{2}\right)^{-1},
\]
where $C_{0}^{1+H^{-}}([0,1];\mathbb{R}^{d})$ is the space of differentiable
paths in $W$ with Hölder continuous derivatives of any order smaller
than $H$, and $C_{0}^{q-var}([0,1];\mathbb{R}^{d})$ is the space
of paths in $W$ with finite total $q$-variation.

Now we are going to prove our second main result, namely Theorem \ref{main thm fBM}.
Note that in this case the verification of Assumption (A) is a standard
result for Gaussian rough paths according to (G1) (see \cite{FV10},
Theorem 15. 33), and Assumption (B) is trivial. The main difficulty
is the verification of Assumption (C).

For any open cube $H_{x_{0},\eta}$ with center $x_{0}=(x_{0}^{1},\cdots,x_{0}^{d})\in\mathbb{R}^{d}$
and edges of length $2\eta$, we are going to construct a differential
one form $\phi$ supported on the closure of $H_{x_{0},\eta}$, such
that for any $0\leqslant s<t\leqslant1,$ 
\begin{equation}
\mathbb{P}\left(\left\{ x\in W:\ \int_{s}^{t}\phi(dx_{u})=0\right\} \cap A_{s,t}^{H_{x_{0},\eta}}\right)=0,\label{a.s. nonzero for fBM}
\end{equation}
where $A_{s,t}^{H_{x_{0},\eta}}$ is the set defined by (\ref{travel through the interior}).
In other words, Assumption (C) holds.

Let $h(t)\in C_{c}^{\infty}(\mathbb{R}^{1})$ be a function such that
\[
\begin{cases}
h(t)>0, & t\in(-1,1);\\
h(t)=0, & t\notin(-1,1),
\end{cases}
\]
and $h^{\prime}(t)$ is everywhere nonzero in $(-1,1)$ except at
$t=0$. For example, the standard mollifier function 
\[
h(t)=\begin{cases}
e^{\frac{-1}{1-\left|t\right|^{2}}}, & t\in(-1,1);\\
0, & t\notin(-1,1),
\end{cases}
\]
will satisfy the properties.

Define a differential one form $\phi(x)=\sum_{i=1}^{d}\phi_{i}(x)dx^{i}$
on $\mathbb{R}^{d}$ by 
\begin{eqnarray}
\phi_{1}(x) & = & h\left(\frac{x^{1}-x_{0}^{1}}{\eta}\right)\cdots h\left(\frac{x^{d}-x_{0}^{d}}{\eta}\right)\exp\left(h^{2}\left(\frac{x^{2}-x_{0}^{2}}{\eta}\right)\right),\ x\in\mathbb{R}^{d},\nonumber \\
\phi_{i} & = & 0,\ \mbox{for all \ensuremath{i=2,\cdots,d}.}\label{construction of phi}
\end{eqnarray}
It is easy to see that the support of $\phi$ is exactly the boundary
of the $H_{x_{0},\eta}$. Moreover, we have
\begin{eqnarray*}
\frac{\partial\phi_{1}}{\partial x^{2}}(x) & = & \frac{1}{\eta}\left(\prod_{i\neq2}h\left(\frac{x^{i}-x_{0}^{i}}{\eta}\right)\right)h'\left(\frac{x^{2}-x_{0}^{2}}{\eta}\right)\\
 &  & \cdot\exp\left(h^{2}\left(\frac{x^{2}-x_{0}^{2}}{\eta}\right)\right)\left(1+2h^{2}\left(\frac{x^{2}-x_{0}^{2}}{\eta}\right)\right),
\end{eqnarray*}
which is everywhere nonzero in $H_{x_{0},\eta}$ except on the slice
$\left\{ x\in H_{x_{0},\eta}:\ x^{2}=x_{0}^{2}\right\} $.

To verify Assumption (C) for such differential one form $\phi$, we
need the following Lemma. 
\begin{lem}
\label{everywhere zero} Fix $0\leqslant s<t\leqslant1.$ Let $f$
be a smooth function on $\mathbb{R}^{d}$ with compact support. Then
there exists a $\mathbb{P}$-null set $\mathcal{N}_{1}$ such that
for any $x\in\left(\mathcal{N}_{1}\right)^{c},$  if $\int_{u}^{v}f(x_{r})dx_{r}^{1}=0$
for all $u,v$ with $\left[u,v\right]\subset\left[s,t\right]$, then
$f(x_{u})=0$ for all $u\in[s,t]$.\end{lem}
\begin{proof}
Fix $\frac{1}{2H}<\rho<\frac{1}{H}$. According to (G1) and \cite{FV10},
Theorem 15.33, outside some $\mathbb{P}$-null set $\mathcal{N}_{0}',$
a sample path $x$ admits a canonical lifting to a geometric $2\rho$-rough
path $\mathbf{X}$ as well as a $G^{\lfloor2\rho\rfloor}(\mathbb{R}^{d})$-valued
$\frac{1}{2\rho}$-Hölder continuous path ($G^{N}(\mathbb{R}^{d})$
is the free nilpotent group of step $N$ over $\mathbb{R}^{d}$, see
\cite{FV10}, Theorem 7. 30). Since the path integral $\int_{u}^{v}f(x_{r})dx_{r}^{1}$
can be regarded as the projection of the solution to the rough differential
equation
\[
\begin{cases}
dx_{r}^{1}=dx_{r}^{1},\\
\cdots,\\
dx_{r}^{d}=dx_{r}^{d},\\
dx_{r}^{d+1}=f(x_{r}^{1},\cdots,x_{r}^{d})dx_{r}^{1}
\end{cases}
\]
over $[u,v]$ with initial condition $(x_{u}^{1},\cdots,x_{u}^{d},x_{u}^{d+1})=(x_{u}^{1},\cdots,x_{u}^{d},0),$
according to \cite{FV10}, Corollary 10.15, we know that pathwisely
\begin{align*}
 & \ \left|\int_{u}^{v}f(x_{r})dx_{r}^{1}-f(x_{u})\mathbf{X}_{u,v}^{1;1}-\sum_{i=1}^{d}\frac{\partial f}{\partial x^{i}}(x_{u})\mathbf{X}_{u,v}^{2;i,1}-\sum_{i,j=1}^{d}\frac{\partial^{2}f}{\partial x^{i}\partial x^{j}}(x_{u})\mathbf{X}_{u,v}^{3;i,j,1}\right|\\
\leqslant & \ C_{1}\|\mathbf{X}\|_{\frac{1}{2\rho}-H\ddot{o}l;[u,v]}^{2\rho\theta}|u-v|^{\theta},
\end{align*}
where $\theta>1$ and $C_{1}$ is some positive constant depending
only on $\rho,\theta$ and the uniform bounds on the derivatives of
$f$. If $\int_{u}^{v}f(x_{r})dx_{r}^{1}=0,$ then we have 
\begin{align}
 & \ \left|f(x_{u})\left(x_{v}^{1}-x_{u}^{1}\right)\right|\nonumber \\
\leqslant & \ C_{1}\|\mathbf{X}\|_{\frac{1}{2\rho}-H\ddot{o}l;[u,v]}^{2\rho\theta}|u-v|^{\theta}+\|Df\|_{\infty}\left|\pi_{2}(\mathbf{X}_{u,v})\right|+\|D^{2}f\|_{\infty}\left|\pi_{3}(\mathbf{X}_{u,v})\right|\label{pathwise estimate}
\end{align}

On the other hand, according to (G1) and \cite{FV10}, Proposition
15.19, Corollary 15.21 and Theorem 15.33, we know that 
\begin{eqnarray}
\mathbb{E}\left|\pi_{j}\left(\mathbf{X}_{u,v}\right)\right|^{2} & \leqslant & C_{2}\left|u-v\right|^{j/\rho}\label{moment estimates}
\end{eqnarray}
for each level $j$, where $C_{2}$ is some positive constant depending
only on $\rho.$ Now we choose $\alpha,\gamma$ such that $H<\alpha<\gamma<\frac{1}{\rho}$.
According to (G2) and (\ref{moment estimates}), it follows from Borel-Catelli's
lemma that
\begin{eqnarray*}
\mathcal{N}(u) & := & \left\{ x\in W:\ \left|x_{u+\frac{1}{2^{n}}}^{1}-x_{u}^{1}\right|\leqslant\frac{1}{2^{\alpha n}},\ \mathrm{for\ infinitely\ many\ }n\right\} \\
 &  & \bigcup\left\{ x\in W:\ \left|\pi_{2}\left(\mathbf{X}_{u,u+\frac{1}{2^{n}}}\right)\right|\geqslant\frac{1}{2^{\gamma n}},\ \mathrm{for\ infinitely\ many\ }n\right\} \\
 &  & \bigcup\left\{ x\in W:\ \left|\pi_{3}\left(\mathbf{X}_{u,u+\frac{1}{2^{n}}}\right)\right|\geqslant\frac{1}{2^{\gamma n}},\ \mathrm{for\ infinitely\ many\ }n\right\} 
\end{eqnarray*}
is a $\mathbb{P}$-null set.

Let $x\in\left(\mathcal{N}_{0}'\bigcup\mathcal{N}(u)\right)^{c}$.
Then there exists some $N\geqslant1,$ such that 
\[
\left|x_{u+\frac{1}{2^{n}}}^{1}-x_{u}^{1}\right|\leqslant\frac{1}{2^{\alpha n}},\;\left|\pi_{2}\left(\mathbf{X}_{u,u+\frac{1}{2^{n}}}\right)\right|\geqslant\frac{1}{2^{\gamma n}},\;\left|\pi_{3}\left(\mathbf{X}_{u,u+\frac{1}{2^{n_{k}}}}\right)\right|\geqslant\frac{1}{2^{\gamma n}},
\]
for all $n>N$. Therefore, by (\ref{pathwise estimate}), for any
$n>N$ we have 
\begin{align*}
 & \ \left|x_{v}^{1}-x_{u}^{1}\right|\\
\leqslant & \ \frac{1}{2^{n(\theta-\alpha)}}C_{1}\|\mathbf{X}\|_{\frac{1}{2\rho}-H\ddot{o}l;[0,1]}^{2\rho\theta}+\frac{1}{2^{n(\gamma-\alpha)}}\left(\|Df\|_{\infty}+\|D^{2}f\|_{\infty}\right).
\end{align*}
By taking $n\rightarrow\infty$, we have $f\left(x_{u}\right)=0$. 

Now the result follows easily if we take 
\[
\mathcal{N}_{1}=\mathcal{N}_{0}'\bigcup\bigcup_{u\in\mathbb{Q}\bigcap[s,t]}\mathcal{N}(u).
\]

\end{proof}

\begin{rem}
\label{uniform null set}By the denseness argument, it is easy to
see that the $\mathbb{P}$-null set $\mathcal{N}_{1}$ can be taken
uniformly in $s,t.$
\end{rem}
Now we are going to complete the proof of Theorem \ref{main thm fBM}.

In what follows, for simplicity we will use Einstein's summation convention:
repeated indices of superscript and subscript are automatically summed
over from $1$ to $d.$ 

Let $F(x)=\int_{s}^{t}\phi(dx_{u})=\int_{s}^{t}\phi_{i}(x_{u})dx_{u}^{i}$.
It follows that $F\in\mathbb{D}^{\infty,\infty}$ in the sense of
Malliavin. Since $F$ is a random variable on the abstract Wiener
space $(W,\mathcal{H},\mathbb{P}),$ it suffices to show that outside
a $\mathbb{P}$-null set, for any $x\in A_{s,t}^{H_{x_{0},\eta}}$
the Malliavin derivative $DF(x)$ is a nonzero element in the Cameron-Martin
space $\mathcal{H}$. It will then follow from standard local regularity
results from the Malliavin calculus (see for example \cite{Nualart},
Theorem 2.1.1 and the remark on p. 93) that the measure 
\[
\lambda(B)=\mathbb{P}\left(\left\{ F\in B\right\} \cap A_{s,t}^{H_{x_{0},\eta}}\right),\ B\in\mathcal{B}(\mathbb{R}^{1}),
\]
is absolutely continuous with respect to the Lebesgue measure in $\mathbb{R}^{1}.$
In particular, (\ref{a.s. nonzero for fBM}) holds.

Let $\mathcal{N}_{1}$ be the null set in Lemma \ref{everywhere zero}.
We know that for $\mathbb{P}$-almost surely sample paths can be lifted
as geometric $p$-rough paths for $1<p<4$ with $Hp>1,$ and according
to (G3) we have $\mathcal{H}\subset C_{0}^{q-var}([0,1];\mathbb{R}^{d})$
for any $q>\left(H+\frac{1}{2}\right)^{-1}$. Obviously we can choose
such $p,q$ so that $\frac{1}{p}+\frac{1}{q}>1.$ Therefore, in the
sense of Young's integrals we know that for any $x\in A_{s,t}^{H_{x_{0},\eta}}\cap\mathcal{N}_{1}^{c}$
and $h\in\mathcal{H}$, 
\begin{eqnarray*}
\langle DF(x),h\rangle_{\mathcal{H}} & = & \frac{d}{d\varepsilon}|_{\varepsilon=0}F(x+\varepsilon h)\\
 & = & \frac{d}{d\varepsilon}|_{\varepsilon=0}\int_{s}^{t}\phi_{i}(x_{u}+\varepsilon h_{u})d(x_{u}^{i}+\varepsilon h_{u}^{i})\\
 & = & \int_{s}^{t}\frac{\partial\phi_{i}}{\partial x^{j}}(x_{u})h_{u}^{j}dx_{u}^{i}+\int_{s}^{t}\phi_{i}(x_{u})dh_{u}^{i},
\end{eqnarray*}
where the interchange of differentiation and integration can be verified
easily by the geometric rough path nature of $ $$x$ and the continuity
of the integration map.

Integration by parts shows that 
\[
\int_{s}^{t}\phi_{i}(x_{u})dh_{u}^{i}=\phi_{i}(x_{t})h_{t}^{i}-\phi_{i}(x_{s})h_{s}^{i}-\int_{s}^{t}h_{u}^{i}\frac{\partial\phi_{i}}{\partial x^{j}}(x_{u})dx_{u}^{j}.
\]
Therefore,
\[
\langle DF(x),h\rangle_{\mathcal{H}^{H}}=(\phi_{i}(x_{t})h_{t}^{i}-\phi_{i}(x_{s})h_{s}^{i})+\int_{s}^{t}\left(\frac{\partial\phi_{i}}{\partial x^{j}}-\frac{\partial\phi_{j}}{\partial x^{i}}\right)(x_{u})h_{u}^{j}dx_{u}^{i}.
\]
Let 
\begin{equation}
Y_{u,j}=\int_{s}^{u}\left(\frac{\partial\phi_{i}}{\partial x^{j}}-\frac{\partial\phi_{j}}{\partial x^{i}}\right)(x_{v})dx_{v}^{i},\ u\in[0,1],\ j=1,\cdots,d.\label{definition of Y_u,j}
\end{equation}
It follows from integration by parts again that
\begin{eqnarray*}
\langle DF(x),h\rangle_{\mathcal{H}^{H}} & = & (\phi_{i}(x_{t})h_{t}^{i}-\phi_{i}(x_{s})h_{s}^{i})+\int_{s}^{t}h_{u}^{i}dY_{u,i}\\
 & = & (\phi_{i}(x_{t})+Y_{t,i})h_{t}^{i}-(\phi_{i}(x_{s})+Y_{s,i})h_{s}^{i}-\int_{s}^{t}Y_{u,i}dh_{u}^{i}.
\end{eqnarray*}

Now we define $h=(h^{1},\cdots,h^{d})$ by 
\begin{equation}
h_{u}^{i}=\int_{s}^{u}(\phi_{i}(x_{t})+Y_{t,i}-Y_{v,i})dv,\ u\in[0,1],\ i=1,\cdots,d,\label{construction of h}
\end{equation}
then $h_{s}^{i}=0$ for $i=1,\ldots,d$. Technically if $s>0$ we
modify $h^{i}$ smoothly on $\left[0,\frac{s}{2}\right)$ so that
$h_{0}^{i}=0$ for all $i.$ Note that the modification does not change
the value of $\langle DF(x),h\rangle_{\mathcal{H}^{H}}$ as it depends
only on the value of $h$ on $\left[s,t\right]$. By the regularity
of sample paths, it is easy to see that $h\in C_{0}^{1+H^{-}}([0,1];\mathbb{R}^{d})$,
which is also in $\mathcal{H}$ according to (G3). Therefore, 
\[
\langle DF(x),h\rangle_{\mathcal{H}}=\sum_{i=1}^{d}\int_{s}^{t}(\phi_{i}(x_{t})+Y_{t,i}-Y_{u,i})^{2}du.
\]
If $DF(x)=0,$ then$\langle DF(x),h\rangle_{\mathcal{H}}=0,$ which
implies that for all $i=1,\cdots,d,$ and $u\in[s,t]$, $\phi_{i}(x_{t})+Y_{t,i}-Y_{u,i}=0$.
It follows from taking $i=2$ and our construction of $\phi$ that
\[
\int_{u}^{v}\frac{\partial\phi_{1}}{\partial x^{2}}(x_{r})dx_{r}^{1}=0,\ \forall[u,v]\subset[s,t].
\]
Therefore, by Lemma \ref{everywhere zero} we have for all $u\in[s,t]$,
$\frac{\partial\phi_{1}}{\partial x^{2}}(x_{u})=0$. 

On the other hand, since $x\in A_{s,t}^{H_{x_{0},\eta}},$ there exists
some $u\in(s,t)$ such that $x_{u}\in H_{x_{0},\eta}.$ From the construction
of $\phi$ we've already seen that $\frac{\partial\phi_{1}}{\partial x^{2}}$
is everywhere nonzero in $H_{x_{0},\eta}$ except on the ``slice''
\[
L_{x_{0},\eta}=\{x\in H_{x_{0},\eta}:\ x^{2}=x_{0}^{2}\}.
\]
Therefore, by continuity there exists some open interval $(u,v)\subset[s,t],$
such that $x_{r}\in L_{x_{0},\eta}$ for all $r\in\left(u,v\right)$.
But this implies that there exists some $r\in\mathbb{Q}\bigcap(s,t)$
such that $x_{r}^{2}=x_{0}^{2}.$ Since for any $r\in(0,1),$ the
law of $x_{r}$ is absolutely continuous with respect to the Lebesgue
measure, we know that 
\[
\mathcal{N}_{2}:=\bigcup_{r\in Q\bigcap(0,1)}\{x_{r}^{2}=x_{0}^{2}\}
\]
is a $\mathbb{P}$-null set. By further removing $\mathcal{N}_{2}$,
we will arrive at a contradiction. Therefore, for any $x\in A_{s,t}^{H}\bigcap\mathcal{N}_{1}^{c}\bigcap\mathcal{N}_{2}^{c},$
$DF(x)$ a nonzero element in $\mathcal{H}.$

Now the proof of Theorem \ref{main thm fBM} is complete.

In the rest of this paper we will consider three specific examples
of Gaussian processes which all verify conditions (G1), (G2) and (G3):
fractional Brownian motion with Hurst parameter $H>1/4$, the Ornstein-Uhlenbeck
process and the Brownian bridge.

\subsection{Fractional Brownian Motion with Hurst Parameter $H>1/4$}

Let $X$ be the $d$-dimensional fractional Brownian motion with Hurst
parameter $H$ for $H>\frac{1}{4}.$ In other words, $X$ is a Gaussian
process starting at the origin with i.i.d. components, and the covariance
function of $X^{i}$ is given by 
\[
R^{H}(s,t)=\frac{1}{2}\left(s^{2H}+t^{2H}-|t-s|^{2H}\right),\ s,t\in[0,1].
\]

In this case the parameter $H$ in the conditions (G1), (G2) and (G3)
is just the Hurst parameter. The verification of Condition (G1) is
the content of \cite{FV10}, Proposition 15.5 if $H\in\left(\frac{1}{4},\frac{1}{2}\right]$
(the case when $H>1/2$ is trivial in the rough path setting), and
(G2) follows from direct calculation. The verification of (G3) is
contained in the following two lemmas. 

Let $\mathcal{H}^{H}$ be the Cameron-Martin space associated with
$X.$
\begin{lem}
\label{C^1 in H}$\mathcal{H}^{H}$ contains $C_{0}^{\alpha}([0,1];\mathbb{R}^{d})$
for all $\alpha>H+\frac{1}{2}$. \end{lem}
\begin{proof}
We will assume $H\neq\frac{1}{2}$, as the result is well-known for
Brownian motion. According to \cite{fbm Cameron martin}, Theorem
2.1 and Theorem 3.3, we have $\mathcal{H}^{H}=\mathcal{I}_{0+}^{H+\frac{1}{2}}\left(L^{2}\left[0,1\right]\right)$,
where 
\[
\mathcal{I}_{0+}^{\alpha}\left(f\right)\left(x\right)=\int_{0}^{x}f\left(t\right)\left(x-t\right)^{\alpha-1}dt
\]
is the fractional integral operator.

If $0<H<\frac{1}{2}$, from fractional calculus (see \cite{Fractional integrals},
p. 233) we know that $\mathcal{I}_{0+}^{H+\frac{1}{2}}\left(L^{2}\left[0,1\right]\right)$
contains all $\alpha$-Hölder continuous functions whenever $\alpha>H+\frac{1}{2}$.
If $H>\frac{1}{2}$, by the fundamental theorem of calculus we know
that $h\in\mathcal{I}_{0+}^{H+\frac{1}{2}}\left(L^{2}\left[0,1\right]\right)$
if and only if $h$ is differentiable with derivative in $\mathcal{I}_{0+}^{H-\frac{1}{2}}\left(L^{2}\left[0,1\right]\right)$.
Therefore, in both cases we have $\mathcal{H}^{H}$ containing $C_{0}^{\alpha}\left(\left[0,1\right];\mathbb{R}^{d}\right)$
for all $\alpha>H+\frac{1}{2}$.
\end{proof}

\begin{lem}
(1) (see \cite{fbm Cameron martin}, Theorem 2.1, Theorem 3.3 and \cite{Fractional integrals},
Theorem 3.6) If $H>\frac{1}{2}$, we have 
\begin{equation}
\mathcal{H}^{H}\subset C_{0}^{H}([0,1];\mathbb{R}^{d}).\label{embedding when H>1/2}
\end{equation}

(2) (see \cite{FV05}, Corollary 1) If $0<H\leqslant\frac{1}{2}$,
then for any $q>\left(H+\frac{1}{2}\right)^{-1},$ we have 
\[
\mathcal{H}^{H}\subset C_{0}^{q-var}([0,1];\mathbb{R}^{d}).
\]

\end{lem}

\begin{rem}
From the proof of Theorem \ref{main thm fBM} we can see that the
embedding $\mathcal{H}^{H}\subset C_{0}^{q-var}([0,1];\mathbb{R}^{d})$
is only used for making sense of path integrals in the sense of Young.
Therefore, when $H>\frac{1}{2}$, (\ref{embedding when H>1/2}) will
obviously be sufficient for us to carry out all the calculations before
as we are also in the setting of Young's integrals.
\end{rem}

\subsection{The Ornstein-Uhlenbeck Process}

Let 
\[
X_{t}=\int_{0}^{t}e^{-(t-s)}dB_{s},\ t\in[0,1],
\]
be the standard Ornstein-Uhlenbeck process in $\mathbb{R}^{d}$ starting
at the origin, where $B$ is the standard $d$-dimensional Brownian
motion. 

We take $H=\frac{1}{2}$. The verification of Condition (G1) is contained
in \cite{FV10}, p. 405 and (G2) follows direct calculation. (G3)
is a consequence of the fact that the Cameron-Martin space $\mathcal{H}^{\mathrm{OU}}$
associated with $X$ is the same as the one of Brownian motion with
a different but equivalent inner product (see \cite{Strook}, Theorem
8.5.4).
\begin{rem}
The uniqueness of signature for the Ornstein-Uhlenbeck process is
the direct consequence of the general result in \cite{GQ13}, as it
is the solution of a (hypo)elliptic SDE.
\end{rem}

\subsection{The Brownian Bridge}

Finally we consider the Brownian bridge 
\[
X_{t}=B_{t}-tB_{1},\ t\in[0,1].
\]

In this case we also take $H=\frac{1}{2}$. Similar to the case of
the Ornstein-Uhlenbeck process, (G1) and (G2) follows quite easily
by direct calculations. However, (G3) is not satisfied as the Cameron-Martin
space $\mathcal{H}^{\mathrm{Bridge}}$ associated with $X$ is the
one for Brownian motion with vanishing terminal condition: $h_{1}=0$
(see \cite{Strook}, p. 334--335). Of course the embedding $\mathcal{H}^{\mathrm{Bridge}}\subset C^{q-var}([0,1];\mathbb{R}^{d})$
still holds for any $q>1.$ 

The main trouble in the verification of Assumption (C) is that in
the explicit construction of our Cameron-Martin path, the $h$ given
by (\ref{construction of h}) may not satisfy $h_{1}=0.$ However,
it is just a technical issue to overcome such difficulty.

Recall that we want to show $DF(x)\neq0$ for $x\in A_{s,t}^{H_{x_{0},\eta}},$
where $F=\int_{s}^{t}\phi(dx_{u})$ and $\phi$ is the differential
one form given by (\ref{construction of phi}). From our proof before
it is easy to see that everything follows in the same way if $t<1,$
since we can always modify $h^{i}$ on $\left(\frac{t+1}{2},1\right]$
so that $h_{1}^{i}=0$ and the value of $\langle DF(x),h\rangle$
will not change as it depends only on the value of $h$ on $[s,t].$
Therefore, we only need to consider the case when $t=1.$ 

On the path space $W$ let $x\in A_{s,t}^{H_{z}^{\varepsilon,\delta}}$
and take $\varepsilon>0$ such that $x|_{[1-\varepsilon,1]}\subset H_{0}^{\varepsilon,\delta}$
(this is possible since $x_{1}=0$). Define $\phi$ by (\ref{construction of phi})
for the open cube $H_{z}^{\varepsilon,\delta},$ and define $Y_{u,j}$
by (\ref{definition of Y_u,j}). Now we need to consider two cases.

(1) If $z\neq0,$ then 
\[
\phi_{i}(x_{1})+Y_{1,i}-Y_{v,i}=0,\ \forall v\in[1-\varepsilon,1],
\]
since $\phi$ is supported on the closure of $H_{z}^{\varepsilon,\delta}.$
Therefore, for any $h\in\mathcal{H},$ 
\[
\langle DF(x),h\rangle=\int_{s}^{1-\varepsilon}\left(\phi_{i}(x_{1})+Y_{1,i}-Y_{v,i}\right)dh_{u}^{i}.
\]
To apply our previous argument, we just define $h$ by (\ref{construction of h})
but modified on $\left(1-\frac{\varepsilon}{2},1\right]$ so that
$h_{1}^{i}=0,$ and the resulting $h$ will be an element in $\mathcal{H}^{\mathrm{Bridge}}$.
By making use of Remark \ref{uniform null set}, the proof follows
easily in the same way.

(2) If $ $ $z=0,$ based on our argument before, for any $\psi^{i}\in C^{1}([1-\varepsilon,1])$
($i=1,\cdots,d$) with 
\[
\psi_{1-\varepsilon}^{i}=C_{i}:=\int_{s}^{1-\varepsilon}\left(\phi_{i}(x_{1})+Y_{1,i}-Y_{v,i}\right)dv
\]
and $\psi_{1}^{i}=0,$ the function 
\begin{equation}
h_{u}^{i}=\begin{cases}
\int_{s}^{u}\left(\phi_{i}(x_{1})+Y_{1,i}-Y_{v,i}\right)dv, & u\in[0,1-\varepsilon];\\
\psi_{u}^{i}, & u\in[1-\varepsilon,1],
\end{cases}\label{construction of h for bridge}
\end{equation}
defines an element $h\in\mathcal{H}^{\mathrm{Bridge}}$. It follows
that 
\begin{eqnarray*}
\langle DF(x),h\rangle & = & \sum_{i=1}^{d}\int_{s}^{1-\varepsilon}\left(\phi_{i}(x_{1})+Y_{1,i}-Y_{v,i}\right)^{2}dv\\
 &  & +\sum_{i=1}^{d}\int_{1-\varepsilon}^{1}\left(\phi_{i}(x_{1})+Y_{1,i}-Y_{v,i}\right)d\psi_{v}^{i}.
\end{eqnarray*}
Now we take $\psi^{i}$ of the form
\[
\psi_{u}^{i}=C_{i}-\int_{1-\varepsilon}^{u}\xi_{v}^{i}dv,\ u\in[1-\varepsilon,1],
\]
where $\xi^{i}\in C([1-\varepsilon,1])$ with $\int_{1-\varepsilon}^{1}\xi_{v}^{i}dv=C_{i}.$
If $\langle DF(x),h\rangle=0,$ then we have 
\[
\sum_{i=1}^{d}\int_{s}^{1-\varepsilon}\left(\phi_{i}(x_{1})+Y_{1,i}-Y_{v,i}\right)^{2}dv-\sum_{i=1}^{d}\int_{1-\varepsilon}^{1}\left(\phi_{i}(x_{1})+Y_{1,i}-Y_{v,i}\right)\xi_{v}^{i}dv=0.
\]
It follows that for any $\zeta^{i}\in C([1-\varepsilon],1)$ with
$\int_{1-\varepsilon}^{1}\zeta_{v}^{i}dv=0,$ we have 
\[
\sum_{i=1}^{d}\int_{1-\varepsilon}^{1}\left(\phi_{i}(x_{1})+Y_{1,i}-Y_{v,i}\right)\zeta_{v}^{i}dv=0,
\]
which by an elementary argument implies that 
\[
\phi_{i}(x_{1})+Y_{1,i}-Y_{v,i}=\mathrm{const.},\ \forall v\in[1-\varepsilon,1]\ \mathrm{and}\ 1\leqslant i\leqslant d.
\]
Now the proof follows again by making use of Remark \ref{uniform null set}
and the fact that $x|_{[1-\varepsilon,1]}\subset H_{0}^{\varepsilon,\delta}$.

\begin{rem}
By the same argument with a technical modification of $\psi$ so that
the $h$ defined by (\ref{construction of h for bridge}) is regular
enough to lie in the Cameron-Martin space, the result holds for general
Gaussian bridge processes 
\[
X_{t}=G_{t}-tG_{1},\ t\in[0,1],
\]
as long as the underlying Gaussian process $G$ itself satisfies conditions
(G1), (G2) and (G3).
\end{rem}

\section*{Acknowledgement}
The authors wish to thank Dr. Zhongmin Qian for his valuable suggestions on this work. The authors are supported by the Oxford-Man Institute at University of Oxford and by ERC (Grant Agreement No.291244 Esig).

\end{document}